\documentclass[a4paper]{article}
\pdfoutput=1
\usepackage{arxiv}

\usepackage[T1]{fontenc}
\usepackage[ascii]{inputenc}
\usepackage[english]{babel}
\usepackage{tikz-cd}

\usepackage{pdfsync}

\usepackage{cite}
\usepackage{nameref,hyperref,url}

\usepackage{amsmath,amsfonts,amsthm,amssymb}
\usepackage{todonotes}

\usepackage{booktabs}       %
\usepackage{nicefrac}       %
\usepackage{microtype}      %
\usepackage{doi}

\allowdisplaybreaks[2]

\newtheorem{theorem}{Theorem}[section]

\newtheorem{proposition}[theorem]{Proposition}

\theoremstyle{remark}
\newtheorem{remark}[theorem]{Remark}
\theoremstyle{definition}
\newtheorem{definition}{Definition}[section]
\newtheorem{example}{Example}[section]

\newcommand{\R}{\mathbb{R}}
\newcommand{\N}{\mathbb{N}}
\newcommand{\Rn}{\R^n}
\newcommand{\Sn}{\mathbb{S}^{n-1}}
\renewcommand{\O}{\Omega}

\newcommand{\Od}{\O_{\delta}}
\newcommand{\Wd}{\omega_{\delta}}
\newcommand{\dx}{\,\mathrm{d}x}

\newcommand{\grad}{\ddot{\mathcal{G}}_{\delta}}

\newcommand{\diver}{\mathcal{D}_{\delta}}

\newcommand{\U}{\mathbb{U}}
\newcommand{\Uz}{\U_0}

\newcommand{\Udz}{\U_{\delta,0}}

\newcommand{\Q}{\mathbb{Q}}
\newcommand{\Qd}{\mathbb{Q}_{\delta}}
\newcommand{\Qds}{\Qd^{\text{a}}}

\makeatletter
\newcommand{\thickbar}{\mathpalette\@thickbar}
\newcommand{\@thickbar}[2]{{#1\mkern1.5mu\vbox{
  \sbox\z@{\(#1\mkern-1.5mu#2\mkern-1.5mu\)}%
  \sbox\tw@{\(#1\overline{#2}\)}%
  \dimen@=\dimexpr\ht\tw@-\ht\z@-.8\p@\relax
  \hrule\@height.8\p@ 
  \vskip\dimen@
  \box\z@}\mkern1.5mu}
}
\makeatother

\newcommand{\Qdx}{\thickbar{\mathbb{Q}}_{\delta}}

\newcommand{\diverx}{\thickbar{\mathcal{D}}_{\delta}}

\newcommand{\knl}{{\ddot{\kappa}}}
\newcommand{\qnl}{{\ddot{q}}}
\newcommand{\pnl}{{\ddot{p}}}

\newcommand{\mres}{\mathbin{\vrule height 1.6ex depth 0pt width
0.13ex\vrule height 0.13ex depth 0pt width 1.3ex}}

\newcommand{\cM}{\mathcal{M}(\cl\O;\Rn)}

\newcommand{\Id}{{\hat{I}}}
\newcommand{\Ip}{{\check{I}}}
\newcommand{\id}{{i}_{\delta}}
\newcommand{\idloc}{i_{\text{loc}}}
\newcommand{\Idloc}{\Id_{\text{loc}}}
\newcommand{\Iploc}{\Ip_{\text{loc}}}

\newcommand{\ind}{\mathbb{I}}

\newcommand{\Jdb}{J_{\delta,\beta}}

\newcommand{\ad}{{\hat{a}}}
\newcommand{\ap}{{\check{a}}}

\newcommand{\kadm}{\mathbb{A}}
\newcommand{\kadmd}{\kadm_{\delta}}

\newcommand{\bite}{\overset{b}{\rightharpoonup}}
\newcommand{\wsto}{\overset{*}{\rightharpoonup}}

\DeclareMathOperator{\cl}{cl}
\DeclareMathOperator{\innre}{int}

\DeclareMathOperator{\graph}{graph}
\DeclareMathOperator{\Div}{div}

\DeclareMathOperator*{\argmin}{arg\,min}

\DeclareMathOperator*{\esssup}{ess\,sup}

\newcommand{\Divx}{\overline{\Div}}

\hypersetup{
pdftitle={Nonlocal basis pursuit: Nonlocal optimal design of conductive domains in the vanishing material limit},
pdfauthor={Anton Evgrafov, Jose C. Bellido},
pdfkeywords={Optimal design, nonlocal diffusion, mixed variational principles},
pdfsubject={MSC2010 classification: 49J21, 49J45, 49J35, 80M50},
}

\title{Nonlocal basis pursuit: Nonlocal optimal design of conductive domains in
the vanishing material limit}

\author{\href{https://orcid.org/0000-0002-3987-7745}{\includegraphics[scale=0.06]{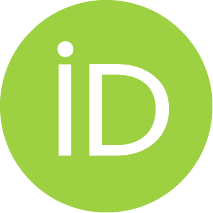}\hspace{1mm}Anton Evgrafov}\\
Department of Mathematical Sciences\\
Aalborg University\\
DK--9210 Aalborg, Denmark\\
\texttt{anev@math.aau.dk}\\
\And
\href{https://orcid.org/0000-0001-5750-1349}{\includegraphics[scale=0.06]{orcid.pdf}\hspace{1mm}Jos\'e C. Bellido}\\
Department of Mathematics\\
University of Castilla-La Mancha\\
13071--Ciudad Real, Spain\\
\texttt{JoseCarlos.Bellido@uclm.es}\\
}

\renewcommand{\shorttitle}{Nonlocal basis pursuit}

\hypersetup{
pdftitle={Nonlocal basis pursuit: Nonlocal optimal design of conductive domains in the vanishing material limit},
pdfauthor={Anton Evgrafov, Jose C. Bellido},
pdfkeywords={Optimal design, nonlocal diffusion, mixed variational principles},
pdfsubject={MSC2010 classification: 49J21, 49J45, 49J35, 80M50},
}

\begin{document}
\maketitle

\begin{abstract}
  We consider the problem of optimal distribution of a limited amount of conductive material in systems governed by local and non-local scalar diffusion laws.
  Of particular interest for these problems is the study of the limiting case, which appears when the amount of available material is driven to zero.  Such a limiting process is of both theoretical and practical interest and continues to be a subject of active study.  In the local case, the limiting optimization problem is convex and has a well understood basis pursuit structure.  Still this local problem is quite challenging both analytically and numerically because it is posed in the space of vector-valued Radon measures.

  With this in mind we focus on identifying the vanishing material limit for the corresponding nonlocal optimal design problem.  Similarly to the local case, the resulting nonlocal problem is convex and has the basis pursuit structure in terms of nonlocal antisymmetric two-point fluxes.  In stark contrast with the local case, the nonlocal problem admits solutions in Lebesgue spaces with mixed exponents.  When the nonlocal interaction horizon is driven to zero, the ``vanishing material limit'' nonlocal problems provide a one-sided estimate for the corresponding local measure-valued optimal design problem.

  The surprising fact is that in order to transform the one-sided estimate into a true limiting process it is sufficient to disregard the antisymmetry requirement on the two-point fluxes.  This result relies on duality and requires generalizing some of the well known nonlocal characterizations of Sobolev spaces to the case of mixed Lebesgue exponents.
\end{abstract}



\keywords{Control in the coefficients \and nonlocal diffusion \and Lebesgue spaces with mixed exponents \and biting convergence \and Fenchel duality}

\noindent  \paragraph*{MSC2010 classification:} 49J21, 49J45, 49J35, 80M50


\section{Introduction}\label{sec:intro}

Let \(\O\) be a connected bounded open domain in \(\R^n\), \(n \geq 2\) with Lipschitz boundary \(\Gamma=\partial\Omega\).
We assume that \(\Gamma\) is subdivided into two open disjoint parts \(\Gamma_D\) and \(\Gamma_N\), on which the two different types of
boundary conditions (Dirichlet and Neumann) will be prescribed.
Let further \(\cM\) be the space of finite vector-valued Radon measures on the closure \(\cl \O\) of \(\O\).
Let us consider the following convex optimization problem, which can be thought of as the analogue of the classical Michell's problem,
see~\cite{michell1904a, strang1983hencky, bendsoe1993michell,
allaire1993optimal, bouchitte2001characterization, bouchitte2008michell, Olbermann2017,
evgrafov2020sparse, babadjian2021shape, bendsoe2013topology,
cherkaev2012variational, allaire2012shape,
bolbotowski2022optimal}, for the case of scalar diffusion:
\begin{equation}\label{eq:limiting_local_dual_rig}
  p^*_{\text{loc}} = \inf_{q\in \cM}[ |q|(\cl\O) + \ind_{\{F\}}(\Divx q)].
\end{equation}
In the problem above,  \(q\in \cM\) is the heat flux (in the heat conduction terminology which we will utilize throughout the manuscript)
corresponding to the heat source \(F \in \mathcal{M}(\cl\O;\R)\),
\(|q|(\cl\O)\) is the total variation of \(q\) over \(\cl\O\), and
\[\ind_S(x) = \begin{cases}0,&\qquad x \in S\\ +\infty,&\qquad x \not\in S,\end{cases}\] is the indicator function associated with a set \(S\).
Furthermore, \(\Divx\) is the divergence operator, which is defined to satisfy the integration by parts formula:
\begin{equation}\label{measdiv}
 \int_{\cl\O} \phi(x)\,\mathrm{d}[\Divx q(x)] = -\int_{\cl\O} \nabla\phi(x)\cdot\,\mathrm{d}q(x),
 \quad \forall \phi \in C^1(\cl\O), \phi|_{\Gamma_D} = 0.
\end{equation}
The most typical situation arises when \begin{equation}\label{volheatsource}F = f\mathcal{H}^{n}\mres \O + g\mathcal{H}^{n-1}\mres \Gamma_N,\end{equation} where
\(\mathcal{H}^{k}\) is the \(k\)-dimensional Hausdorff measure,
 and \(f \in L^1(\O)\) and \(g \in L^1(\Gamma_N)\) are the volumetric and boundary heat sources.

Note that~\eqref{eq:limiting_local_dual_rig} amounts to finding a sparse solution, as given by the smallest possible non-smooth sparsity encouraging norm \(|q|(\cl\O)\), to an underdetermined system of linear equations \(\Divx q = F\).
From this perspective~\eqref{eq:limiting_local_dual_rig} is a basis pursuit problem~\cite{boyd2004convex}.
Problems of this class are well known in the context of robustly recovering sparse solutions in many applied scientific disciplines.
We refer the interested reader to ``LASSO'' (least absolute shrinkage and selection operator) methods in statistics~\cite{santosa1986linear,tibshirani1996regression}, robust denoising of images~\cite{chambolle2011first}, or sparse filtering in compressed sensing~\cite{davenport2013fundamentals}.

 As we shall recall in Section~\ref{subsec:lvan} and~\ref{sec:analysis_loc}, problem~\eqref{eq:limiting_local_dual_rig} arises naturally as a limit of problems of optimally distributing conductive material in \(\Omega\), when the amount of available material is driven to zero, which is a well known fact in the existing
  literature~\cite{strang1983hencky,bendsoe1993michell,allaire1993optimal,bouchitte2001characterization,
  bouchitte2008michell,Olbermann2017,babadjian2021shape,
  bendsoe2013topology,cherkaev2012variational,allaire2012shape}.
With such an interpretation, each \(q\in \cM\) solving~\eqref{eq:limiting_local_dual_rig} is a limiting heat flux, from which approximately optimal conductivity distributions corresponding to vanishingly small but positive material amounts can be easily reconstructed.
The possibility of such a reconstruction explains the continuing interest in~\eqref{eq:limiting_local_dual_rig}, and even more so in its linear elasticity ``cousin,''
both within engineering and applied mathematics communities.
One can go as far as to say that this is one of the most important open problems in the area of optimal design nowadays, whose  efficient solution would lead to a method for computing lines of principal action in optimal elastic structures.  All this motivates us to look at a novel nonlocal approach to this important problem.

The main objective of this study is to derive a natural nonlocal analogue of the basis pursuit problem~\eqref{eq:limiting_local_dual_rig} by computing a  vanishing material limit for a family of recently introduced nonlocal optimal design problems for scalar diffusion, see~\cite{andres2015nonlocal,andres2015type,andres2017convergence,AnMuRo19,evgrafov2019non,AnMuRo21,eb2021}.
We will also study the relation of this limiting nonlocal problem with its local analogue~\eqref{eq:limiting_local_dual_rig}, as the parameter quantifying the nonlocality of the problem, called nonlocal horizon, is driven to zero.

We will show that unlike~\eqref{eq:limiting_local_dual_rig}, its nonlocal analogue admits solutions in certain function spaces, namely in Lebesgue spaces with mixed exponents.
We will also outline a small modification to the limiting nonlocal problem, which allows us to obtain their convergence  towards~\eqref{eq:limiting_local_dual_rig} in the appropriately defined sense.

For the remainder of this paper, primarily to avoid dealing with nonlocal Neumann boundary conditions, we will focus on the Dirichlet problem only and assume \(\Gamma_D=\Gamma\) and \(\Gamma_N=\emptyset\).
In particular, the test space in~\eqref{measdiv} is \(C^1_0(\cl\O)\).
Furthermore, we will assume that the volumetric heat source \(f\), see~\eqref{volheatsource}, is in \(L^2(\O)\).

The outline of the paper is graphically presented in Figure~\ref{roadmap}, and is as follows.
In Section~\ref{sec:statement} we recall and summarize the relevant results for the local and nonlocal primal and dual optimal design problems corresponding to a positive amount of available conductive material.  We then formally derive the limiting problems corresponding to the zero amount of available material. The novel contribution in this section is Subsection~2.4, where we formally derive the limiting nonlocal problem~\eqref{eq:prob_nl}, which should be viewed as the nonlocal analogue of~\eqref{eq:limiting_local_dual_rig}.
In Section~\ref{sec:analysis_loc} we provide a proof of the fact that the previously computed formal vanishing material limit in the local case is in fact also a rigorous limit in the appropriate sense.  Whereas these results are not new, we present them in order to keep the paper self-contained, as well as to preview the techniques which will be later applied to the nonlocal case. Namely, our approach is that we sandwich the problem for a positive amount of material between the vanishing material model and its \(L^2\)-Tikhonov regularization for the appropriately selected regularization parameter, see Theorem~\ref{thm:vanish_limit_local}.  We also rely heavily on Fenchel duality, as some calculations are much more straightforward for the dual problems.
Section~\ref{sec:analysis} and~\ref{sec:gammaconv3} contain the main novel contributions of the paper.
In Subsection~\ref{sec:existence} we establish the existence of solutions for the nonlocal vanishing material problem in Lebesgue spaces with mixed exponents, see Theorem~\ref{thm:existence}. This result relies on the possibility of strengthening the so-called biting convergence in the functional space we work in by utilizing the antisymmetry of the two-point fluxes, see Propositions~\ref{prop:compact} and~\ref{prop:divcont}. Subsections~\ref{sec:fenchel_duals}--\ref{sec:gammaconv1} carry out the plan, which has been previously outlined for the local problem in Section~\ref{sec:analysis_loc}, in the nonlocal case, with the main result being Theorem~\ref{thm:vanish_limit}. Finally, Subsection~\ref{sec:gammaconv2} and Section~\ref{sec:gammaconv3} deal with convergence of the nonlocal vanishing material problem, and its modified version, towards the local limit~\eqref{eq:limiting_local_dual_rig}, when the nonlocal horizon \(\delta\) is driven to zero.
The main resut here is Theorem~\ref{thm:nonlocal_approximation}, which relies on the generalization of the nonlocal characterisation of Sobolev functions by Bourgain, Brezis, Mironescu, Ponce, and others, see~\cite{bourgain2001another,ponce2004new,ponce2004estimate}, to Lebesgue spaces with mixed exponents -- an interesting result in its own right, see Proposition~\ref{prop:compactness}.

\begin{figure}[htb]
  \centering
    \begin{center}
  \begin{tikzpicture}[domain=0:2]
    \draw[color=gray,step=.5cm,dashed,very thin] (-0.5,-.5) grid (3,3);
    \draw[very thick,->,color=gray] (-1,0) -- (3.5,0);
    \draw[very thick,->,color=gray] (0,-1) -- (0,3.5);
    \draw (5,0) node[below] {volume of material, \(\gamma\)};
    \draw (0,3.5) node[above] {nonlocal horizon, \(\delta\)};
    \draw (2.5,0) node[circle,fill=gray]{} ;
    \draw (0,0)   node[circle,fill=gray]{} ;
    \draw[very thick, dashed, ->, color=black] (2.25,0) -- (0.25,0.0) {};
    \draw (2.5,2.5) node[circle,fill=gray]{} ;
    \draw[very thick, dashed, ->, color=black] (2.5,2.25) -- (2.5,0.25) {};
    \draw (0,2.5) node[circle,fill=gray]{} ;
    \draw[very thick, dashed, ->, color=black] (2.25,2.5) -- (0.25,2.5) {};
    \draw[very thick, dashed, ->, color=black] (0,2.25) -- (0,0.25) {};
    \draw (-1.75,0.1) node [above]{\eqref{eq:limiting_local_dual_rig}, \eqref{eq:limiting_local_predual}, \eqref{eq:limiting_local_predual_relax}};
    \draw (3.75,0.1) node [above]{\eqref{eq:compliance_local_primal}, \eqref{eq:compliance_local_dual}};
    \draw (3.9,2.4) node [below]{\eqref{eq:compliance_nl_primal}, \eqref{eq:compliance_nl_mixed}};
    \draw (-2.25,2.4) node [below]{\eqref{eq:prob_nl}, \eqref{eq:limiting_nl_dual_asym}, \eqref{eq:prob_nl_nosym}, \eqref{eq:limiting_nl_dual_nosym}};
    \draw (3.0,1.25) node [right] {Theorem~\ref{thm:primal_summary}};
    \draw (-0.5,1.25) node [left] {Theorem~\ref{thm:nonlocal_approximation}};
    \draw (1.25,-0.5) node [below] {Theorem~\ref{thm:vanish_limit_local}};
    \draw (1.25,3.0) node [above] {Theorem~\ref{thm:vanish_limit}};
    \draw [color=black,opacity=0.3,fill=lightgray]
    (-4,-1) -- (-1,-1) -- (5,5) -- (-4,5) -- cycle;
    \draw (0,5) node [below] {Original contribution of this work};
  \end{tikzpicture}
  \end{center}
  \caption{Graphical roadmap of the paper in the two-parameter space \((\gamma,\delta)\).  Vertices of the diagram correspond to the optimization problems,
  and their labels to the equation numbers defining the relevant problems.
  Edges of the diagram correspond to the limiting processes along the corresponding axes, and their labels to the corresponding results discussing the convergence.  Novel results established in this work are emphasized.}
  \label{roadmap}
\end{figure}
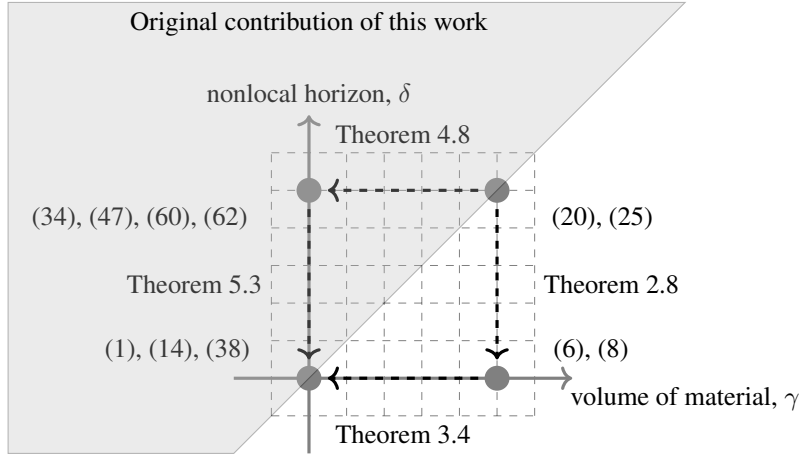

\section{Formal derivation of the limiting nonlocal problem}\label{sec:statement}

\subsection{Local optimal design problems}\label{subsec:local}

Let us start the derivation for the local case by recalling a classical problem, which can be viewed as that of distributing spatially varying heat conductive
material in a given design domain in order to minimize the weighted average steady-state temperature, see~\cite{cea1970example}.
For given constants \(0<\underline{\kappa} < \overline{\kappa}<\infty\) and
\(\gamma \in (\underline{\kappa},\overline{\kappa})\) we define the sets of
admissible conductivities \(\kappa\), temperatures \(u\), and heat fluxes \(q=-\kappa \nabla u\)
by
\begin{equation}
  \begin{aligned}
    \kadm &= \bigg\{\, \kappa \in L^\infty(\O)
    \mid \kappa(x) \in [\underline{\kappa},\overline{\kappa}], \text{\ for almost all \(x\in\O\)},
    \int_{\O} \kappa(x)\dx \leq \gamma|\O|
    \,\bigg\},\\
    \Uz &= \{\, u \in H^1_0(\O) \,\}, \qquad
    \Q = H(\Div,\O;\Rn),\qquad\text{and}\\
    \Q(f)&= \{\, q \in \Q \mid \Div q=f \,\}.
    \end{aligned}
\end{equation}
We note that here and throughout the manuscrupt we use \(|\cdot|\) loosely to denote the
``size'' of objects, such as Lebesgue measure of sets and Eucledian norms of vectors,
with no ambiguity arising from the context.
Recall that the steady state temperature of a solid for a given conductivity
distribution \(\kappa \in \kadm\) can be determined using the Dirichlet principle,
namely
\begin{equation}\label{eq:Dirloc}
  \begin{aligned}
    u &= \argmin_{v\in \Uz} \Iploc(\kappa;v), &\quad&\text{where}\\
    \Iploc(\kappa;v) &= \frac{1}{2}\int_{\O} \kappa(x)|\nabla v(x)|^2\dx - \ell(v), &\quad&\text{and}\\
    \ell(v) &= \int_{\O} f(x) v(x)\dx.
  \end{aligned}
\end{equation}
The associated optimal design problem, see~\cite{cea1970example}, amounts to
finding a distribution of the conductivities, which maximizes \(\Iploc\).
Consequently, this problem can be stated as the following saddle-point problem:
\begin{equation}\label{eq:compliance_local_primal}
  \begin{aligned}
    &\max_{\kappa \in\kadm} \min_{u\in \Uz} \Iploc(\kappa;u).
  \end{aligned}
\end{equation}
As a direct consequence of the concavity and upper semicontinuity of the map
\(\kappa\mapsto \min_{u\in \Uz} \Iploc(\kappa;u)\) it is not difficult to check that~\eqref{eq:compliance_local_primal}
admits an optimal solution.
In a very real sense this problem is a prototype for a wide class of
\emph{topology optimization} problems with applications in all branches of engineering
sciences, see for example~\cite{bendsoe2013topology,cherkaev2012variational,allaire2012shape}.

We now replace Dirichlet's variational principle with that of Kelvin for describing
the steady state of the heated solid.
This variational principle is formulated in terms of the heat flux
\(q \in \Q(f)\) as the primary unknown:
\begin{equation}\label{eq:Kelloc}
  \begin{aligned}
    &q = \argmin_{p\in\Q(f)} \Idloc(\kappa;p),\quad\text{where}\\
    &\Idloc(\kappa;p) = \frac{1}{2}\int_{\O} \kappa^{-1}(x)|p(x)|^2\dx.
  \end{aligned}
\end{equation}
It is not difficult to check that the optimal values for the
problems~\eqref{eq:Dirloc} and~\eqref{eq:Kelloc} are related through the equality \(\Idloc(\kappa;q)+\Iploc(\kappa;u)=0\).
Therefore, the saddle point problem~\eqref{eq:compliance_local_primal}
can be replaced with a single minimization problem with respect to both the
control coefficient \(\kappa\) and the flux variable \(q\):
\begin{equation}\label{eq:compliance_local_dual}
    \min_{(\kappa,q)\in \kadm\times\Q(f)}
    \Idloc(\kappa;q).
\end{equation}
Not surprisingly, in view of its equivalence with~\eqref{eq:compliance_local_primal},
the problem~\eqref{eq:compliance_local_dual} also attains its minimum,
which can also be viewed as a direct consequence of the convexity of the map
\(\R_+\times \Rn\ni(\kappa,q)\mapsto \kappa^{-1}|q|^2\),
see~\cite{bendsoe2013topology,cherkaev2012variational,allaire2012shape}.
Note that the flux-based formulations~\eqref{eq:Kelloc} and~\eqref{eq:compliance_local_dual}
extend to the case \(\underline{\kappa}=0\) without any changes.

\subsection{Vanishing material fraction limit for the local problem}\label{subsec:lvan}

Let us now consider a particular formal limit of~\eqref{eq:compliance_local_dual},
obtained in the vanishing volume fraction case \(\gamma\searrow 0\), as we previously outlined in the introduction.
Let us introduce the scaled conductivity \(\kappa_{\gamma} = \gamma^{-1}\kappa\).
Note that \(\kappa_{\gamma} \in [\gamma^{-1}\underline{\kappa}, \gamma^{-1}\overline{\kappa}]\)
while \(\int_{\O}\kappa_{\gamma} \leq |\Omega|\), which is only possible in the limit
when \(\underline{\kappa} = 0\).
Let us define the admissible set for the scaled conductivities
\begin{equation}\label{eq:kadm_gamma_loc}
  \kadm^\gamma = \bigg\{\, \kappa_\gamma \in L^1(\O)\mid
  0 \le \kappa_\gamma(x) \le \gamma^{-1}\overline{\kappa}, \text{a.e.\ in \(\O\)},
  \int_{\O} \kappa_\gamma(x)\dx \leq |\O|\,\bigg\}.
\end{equation}
Note that \(\Idloc(\gamma^{-1}\kappa;\cdot) = \gamma\Idloc(\kappa;\cdot)\), and therefore
for any \(q \in \Q(f)\) it holds that
\begin{equation}\label{eq:idloc}
  \idloc^\gamma(q):=\min_{\kappa_\gamma \in \kadm^{\gamma}}\Idloc(\kappa_\gamma;q) =
  \gamma\min_{\kappa\in \kadm}\Idloc(\kappa;q).
\end{equation}
Consequently, the optimal design problem~\eqref{eq:compliance_local_dual}, up to
a constant scaling factor \(\gamma\), can be written in terms of fluxes only as
\begin{equation}\label{eq:compl_flux_loc}
  \hat{p}^*_{\text{loc},\gamma} = \min_{q\in\Q(f)} \idloc^\gamma(q).
\end{equation}
Since we are interested in considering a formal limit \(\gamma\searrow 0\),
it is plausible to assume that the upper constraint
\(\kappa_{\gamma}\leq \gamma^{-1}\overline{\kappa}\)
ultimately becomes inactive and therefore will be dropped from~\eqref{eq:kadm_gamma_loc}.
For each \(\gamma>0\) the set \(\kadm^\gamma\) is just a non-empty intersection of a closed ball
\(\{\,\kappa \in L^\infty(\O) \colon
\esssup_{x\in\O}
|\kappa_\gamma(x)- \gamma^{-1}\overline{\kappa}/2|
\leq \gamma^{-1}\overline{\kappa}/2\,\}\), and a weak$^*$-closed in \(L^\infty(\O)\) half-space
\(\{\, \kappa_\gamma \in L^\infty(\O)\colon
\int_{\O} \kappa_\gamma(x)\dx \leq |\O|\,\}\).
It is therefore weak$^*$-sequentially compact in view of Banach--Alaoglu theorem.
Furthermore, the objective function of~\eqref{eq:idloc} is weakly$^*$ lower semicontinuous owing to Tonelli's theorem.
These two facts imply the existence of solutions for the problem~\eqref{eq:idloc} using the direct method of calculus of variations.
Unfortunately, for \(\gamma=0\) we lose the boundedness in \(L^\infty(\O)\) and
have to resort to other tools; also see~\cite{bouchitte2001characterization}.
\begin{proposition}\label{prop:id0}
  For each \(q \in L^1(\O)\) the formal limiting optimization problem
  \begin{equation}\label{eq:formalproblem}
    \min_{\kappa_0 \in \kadm^0}\Idloc(\kappa_0;q)
  \end{equation}
  admits an optimal solution defined pointwise by \(\hat{\kappa}_0(x)=\hat{c}_0 |q(x)|\)
  with the normalization constant \(\hat{c}_0 = |\O|/\|q\|_{L^1(\O;\Rn)}\).
\end{proposition}
\begin{proof}
  The proof can be done for example using a duality argument.
  Indeed, let us put \(\kadm^0_+ = \{\, \kappa_0 \in L^1(\O)\mid \kappa_0(x)\geq 0,
  \text{a.e.\ in \(\O\)}\,\}\).
  We can then write:
  \begin{equation*}
    \begin{aligned}
      \frac{1}{2|\O|}\|q\|_{L^1(\O;\Rn)}^2
      &= \Idloc(\hat{\kappa}_0;q)
      \geq
      \inf_{\kappa_0\in\kadm^0}\Idloc(\kappa_0;q)
      \\&=
      \inf_{\kappa_0\in\kadm^0_+}\sup_{\xi\geq 0}
      \bigg[\Idloc(\kappa_0;q) + \xi\bigg(\int_{\O}\kappa_0(x)\dx-|\O|\bigg)\bigg]
      \\
      &\geq
      \sup_{\xi\geq 0}\inf_{\kappa_0\in\kadm^0_+}
      \bigg[\Idloc(\kappa_0;q) + \xi\bigg(\int_{\O}\kappa_0(x)\dx-|\O|\bigg)\bigg]
      \\
      &=
      \sup_{\xi\geq 0}\begin{cases}
      0,\quad &\text{if\ \(\xi = 0\)},\\
      \Idloc(\hat{\kappa}_{0,\xi};q) + \xi\left(\int_{\O}\hat{\kappa}_{0,\xi}(x)\dx-|\O|\right), \quad &\text{if\ \(\xi > 0\)},
      \end{cases}
      \\&=
      \sup_{\xi\geq 0}\left[\|q\|_{L^1(\O;\Rn)}[2\xi]^{1/2} - \xi|\O|\right]
      =
      \|q\|_{L^1(\O;\Rn)}[2\hat{\xi}]^{1/2} - \hat{\xi}|\O|
      =\frac{1}{2|\O|}\|q\|_{L^1(\O;\Rn)}^2,
    \end{aligned}
  \end{equation*}
  where the inner separable infimum is attained at \(\hat{\kappa}_{0,\xi}(x) = |q(x)|[2\xi]^{-1/2}\) for \(\xi>0\), and the final supremum is attained at
  \(\hat{\xi}=(2|\O|^2)^{-1}\|q\|_{L^1(\O;\Rn)}^2\).
  Therefore, \(\hat{\kappa}_0(x)\) is the optimal solution to~\eqref{eq:formalproblem}.
\end{proof}
With this in mind we evaluate
\(\idloc^0(q)=\Idloc(\hat{c}_0 |q|,q) = (2|\Omega|)^{-1}\|q\|_{L^1(\O;\Rn)}^2\),
and consequently the formal limiting problem corresponding to~\eqref{eq:compliance_local_dual}
is equivalent to
\begin{equation}\label{eq:limiting_local_dual}
  \begin{aligned}
    &\inf_{q}
    \int_{\O} |q(x)|\dx,\\
    &\text{subject to\ } \Div q = f, \quad \text{in \(\O\)}.
  \end{aligned}
\end{equation}
Unfortunately, the infimum in~\eqref{eq:limiting_local_dual} is not necessarily attained in function spaces, such as \(L^1(\O)\), owing to the lack of weak compactness; see also Example~\ref{ex:meas}.
Therefore, we have to further relax this problem in the space \(\cM\) to finally arrive at~\eqref{eq:limiting_local_dual_rig}, which is stated in the introduction.
We note further that~\eqref{eq:limiting_local_dual_rig} is the Fenchel dual of the following problem:
\begin{equation}\label{eq:limiting_local_predual}
  -pd^*_{\text{loc}} = \inf_{u \in C^1_0(\cl\O)} [\ind_{B_{C^0(\cl\O;\Rn)}}(\nabla u)-\ell(u)],
\end{equation}
where throughout the manuscript we will denote by \(B_{X}\) the closed unit ball in a normed space \(X\).
\begin{theorem}\label{thm1}
  Strong Fenchel duality holds, that is, \(p^*_{\text{loc}} = pd^*_{\text{loc}}\).
  Furthermore, whenever \(p^*_{\text{loc}}<+\infty\), the infimum in~\eqref{eq:limiting_local_dual_rig} is attained.
\end{theorem}
\begin{proof}
  We only need to verify that the assumptions for strong Fenchel duality, see for example~\cite[Theorem~4.4.3]{borwein2004techniques},
  hold.
  Indeed:
  \begin{itemize}
    \item both \(\ind_{B_{C^0(\cl\O;\Rn)}}: C^0(\cl\O;\Rn)\to\R\cup\{+\infty\}\) and
    \(\ell: C^1_0(\cl\O;\Rn)\to\R\) are convex and lower semicontinuous;
    \item \(\nabla \in \mathcal{B}(C^1_0(\cl\O),C^0(\cl\O;\Rn))\), that is, \(\nabla\) is a bounded linear operator
    from \(C^1_0(\cl\O)\) to \(C^0(\cl\O;\Rn)\);
    \item we have the inclusion \(0 \in \innre\{B_{C^0(\cl\O;\Rn)}\}\subset \innre\{B_{C^0(\cl\O;\Rn)}-R(\nabla)\}\),
    where \(R(\nabla) = \{\, \phi \in C^0(\cl\O;\Rn)\mid \phi = \nabla u, u \in C^1_0(\cl\O)\,\}\)
    is the range of \(\nabla\).
  \end{itemize}
  Thus strong Fenchel duality holds, and the infimum in~\eqref{eq:limiting_local_dual_rig} is attained whenever
  \(p^*_{\text{loc}}\) is finite.
  However, inserting \(u=0\) into~\eqref{eq:limiting_local_predual} we verify a simple estimate \(p^*_{\text{loc}}= pd^*_{\text{loc}}\geq 0\), which excludes
  the possibility \(p^*_{\text{loc}}=-\infty\), thereby concluding the proof.
\end{proof}
We note that instead of utilizing the strong Fenchel duality argument, the existence
of solutions to~\eqref{eq:limiting_local_dual_rig} can also be established using
the straightforward application of the direct method of calculus of variations.
See for example~\cite{bouchitte2008michell,bouchitte2001characterization,babadjian2021shape,Olbermann2017,bolbotowski2022optimal} for more details on problems of this type.

We would like to conclude this subsection by considering a simple example showing that one generally cannot expect any additional regularity of solutions to~\eqref{eq:limiting_local_dual_rig}, even with very ``nice'' problem's data \(\O\), \(f\), and \(g\).
Note that while the example utilizes mixed (Dirichlet and Neumann) boundary conditions, we hope it allows the reader to further appreciate the difference in the solution regularity for the problems~\eqref{eq:limiting_local_dual_rig} and its nonlocal analogue~\eqref{eq:prob_nl}.
\begin{example}\label{ex:meas}
  Let \(n=2\), \(\O=(0,1)^2\), \(\Gamma_D = \{0\}\times (0,1)\) and \(\Gamma_N\) is the (interior of) the rest of the boundary.
  We prescribe the volumetric heat source \(f\equiv 0\) and a boundary heat source \(g=1\) on \((0,1)\times \{1\}\subset \Gamma_N\), and \(g=0\) otherwise,
  see~\eqref{volheatsource}.
  Let us define the heat flux \(q = (x_1-1,0)^T \mathcal{H}^{n-1}\mres [(0,1)\times \{1\}]\).
  Then, for each \(\phi\in C^1(\cl\O)\), \(\phi|_{\Gamma_D}\equiv 0\) we have
  \[\begin{aligned}\int_{\O\cup\Gamma_N} \phi(x)\,\mathrm{d}[\Divx q(x)]&=-\int_{\cl\O} \nabla\phi(x)\cdot\,\mathrm{d}q(x)
  =-\int_0^1 (x_1-1) \frac{\partial\phi}{\partial x_1}(x_1,1)\,\mathrm{d}x_1
  \\&=\int_0^1 \phi(x_1,1)\,\mathrm{d}x_1.\end{aligned}\]
  Consequently, the flux conservation \(\Divx q=\mathcal{H}^{n-1}\mres [(0,1)\times \{1\}]=f\mathcal{H}^{n}\mres\O + g\mathcal{H}^{n-1}\mres\Gamma_N\)
  holds, and \(q\) is feasible for the modified version of~\eqref{eq:limiting_local_dual_rig}.
  Its optimality may be established utilizing the following arguments.
  On the one hand we have \(p^*_{\text{loc}}\leq|q|(\cl\O)=1/2\).
  On the other hand, from the convex duality considerations we get the following estimate:
  \[\begin{aligned}
  &\frac{1}{2}\geq p^*_{\text{loc}} \\&=
  \inf_{q\in \cM}\sup_{\phi \in C^1(\cl\O), \phi|_{\Gamma_D}=0}
  \bigg[ |q|(\cl\O) - \int_{\cl\O} \nabla \phi(x)\cdot\,\mathrm{d}q(x) + \int_{\Gamma_N} \phi(x)g(x)\,\mathrm{d}x \bigg]
  \\&\geq
  \sup_{\phi \in C^1(\cl\O), \phi|_{\Gamma_D}=0}
  \inf_{q\in \cM}
  \bigg[ |q|(\cl\O) - \int_{\cl\O} \nabla \phi(x)\cdot\,\mathrm{d}q(x) + \int_{\Gamma_N} \phi(x)g(x)\,\mathrm{d}x \bigg]
  \\&\geq
  \sup_{\phi \in C^1(\cl\O), \phi|_{\Gamma_D}=0, |\nabla \phi(\cdot)|\leq 1}
  \inf_{q\in \cM}
  \bigg[ \int_{\cl\O} \nabla \phi(x)\cdot\,\mathrm{d}q(x) - \int_{\cl\O} \nabla \phi(x)\cdot\,\mathrm{d}q(x) + \int_{\Gamma_N} \phi(x)g(x)\,\mathrm{d}x \bigg]
  \\&=
  \sup_{\phi \in C^1(\cl\O), \phi|_{\Gamma_D}=0, |\nabla \phi(\cdot)|\leq 1}
  \bigg[ \int_{\Gamma_N} \phi(x)g(x)\,\mathrm{d}x \bigg] \geq \frac{1}{2},
  \end{aligned}\]
  where the last inequality can be seen by utilizing for example \(\phi(x_1,x_2)=x_1\).
  Therefore, the optimal solution to this instance of~\eqref{eq:limiting_local_dual_rig} is a measure concentrated along the ``heated'' boundary,
  despite the fact that the problem's data is quite regular;
  indeed, \(f\) is a constant, \(g\) is piecewise-constant, and \(\Omega\) is a convex polygon.
  This example can be further modified so that the data is even smoother; the main point being that even starting from the regular problem instance
  we cannot necessarily escape measure-valued solutions to~\eqref{eq:limiting_local_dual_rig}.
\end{example}

\subsection{Nonlocal optimal design problems}

Nonlocal equations and variational problems have attracted a lot of attention recently.
In this spirit, nonlocal versions of the problems~\eqref{eq:compliance_local_primal} and~\eqref{eq:compliance_local_dual}
have been introduced and studied in~\cite{andres2015type,andres2015nonlocal,andres2017convergence,
evgrafov2019sensitivity,evgrafov2019non,eb2021}.
The measure of nonlocality of the problem will be encapsulated in a parameter \(\delta>0\), which is commonly referred to as the nonlocal
interaction horizon.
The set \(\Od = \cup_{x\in \O} B(x,\delta)\), where
\(B(x,\delta) = \{\, z \in \Rn : |z-x|<\delta\,\}\),
is the set of points, which are at most distance \(\delta\) from \(\O\).
These are the points, which are allowed to ``interact'' with points from \(\O\) in the nonlocal model we consider.
The difference set \(\Gamma_{\delta} = \Od\setminus\O\) is thus a ``nonlocal boundary'' of \(\O\).

The strength and character of nonlocal interactions as a function of the distance between the points is represented by a
radial function \(\Wd: \Rn\to \R_+\) with support in \(B(0,\delta)\),
which will be assumed to satisfy the normalization condition
\(\int_{\Rn} |x|^2\Wd^2(x)\dx = K_{2,n}^{-1},\)
where
\[K_{p,n}=\frac{1}{|\Sn|}\int_{\Sn}|s\cdot e|^p\,\mathrm{d}s,\]
\(\Sn\) is the \(n-1\)-dimensional unit sphere,
and \(e\in\Sn\) is an arbitrary unit vector~\cite{bourgain2001another}.
The most canonical example is obtained by setting \(\Wd(x) = c_{\delta,\alpha}|x|^{-(1+\alpha)}\),
where the  constant \(c_{\delta,\alpha}>0\) is determined from the normalization condition above, and \(\alpha \in [0,n/2)\) is a modelling parameter.  More singular behaviour of \(\Wd\) at zero (that is, larger values of \(\alpha\)) results in higher solution regularity to nonlocal diffusion problems.

Throughout the manuscript we adopt the double-dot accent notation to refer to nonlocal two-point quantities.
For example, the nonlocal two-point gradient for a function \(u:\Rn\to\R\) is given by
\begin{equation}\label{eq:defgrad}
  \grad u(x,x') = [u(x)-u(x')]\Wd(x-x'), \quad (x,x')\in \Rn\times\Rn.
\end{equation}
We will also extend all functions by zero outside of their explicit domain of definition, unless specifically stated otherwise.
Having this convention in mind, we will think of \(\grad\) as a linear, possibly unbounded operator \(\grad : D(\grad)\subset L^2(\O) \to L^2(\Od\times\Od)\).
The following properties of this operator will be often utilized in what follows; see for example~\cite[Proposition~3.1]{eb2021}.
\begin{proposition}\label{prop:grad0}
  The following statements hold.
  \begin{enumerate}
    \item The domain of \(\grad\),
    \(\Udz = D(\grad) = \{\,u \in L^2(\O) \mid \grad u \in L^2(\Od\times\Od), u|_{\Gamma_\delta}\equiv 0 \,\}\)
    is dense in \(L^2(\O)\).
    \item \(\ker\grad=\{0\}\).
    \item Poincar{\'e} type inequality holds: \(\exists \overline{\delta}>0\) and \(C>0\), such that for all \(0<\delta < \overline{\delta}\) and \(u\in \Udz\)  we have the inequality \(\|u\|_{L^2(\O)}\leq C\|\grad u\|_{L^2(\Od\times\Od)}\).
    \item The graph of \(\grad\), \(\graph \grad = \{\,(u,\pnl)\in L^2(\O)\times L^2(\Od\times\Od)
    \mid \pnl=\grad u\,\}\) is closed in \(L^2(\O)\times L^2(\Od\times\Od)\).
    \item \(\Udz\) equipped with the inner product \((u_1,u_2)_{\Udz}=(\grad u_1,\grad u_2)_{L^2(\Od\times\Od)}\)
    is a Hilbert space.
    \item The range of \(\grad\), \(R(\grad) = \{\, \pnl \in L^2(\Od\times\Od) \mid \exists
    u \in L^2(\O): \pnl=\grad u\,\}\) is closed in \(L^2(\Od\times\Od)\).
  \end{enumerate}
\end{proposition}

Using the fact that \(\grad\) is closed and densely defined, we can now consider its negative adjoint, which we will refer to as the  non-local linear divergence operator \(\diver : D(\diver)\subset L^2(\Od\times\Od)\to L^2(\O)\).
Thus, \(\diver\) is explicitly defined to satisfy the following integration by parts formula:
\begin{equation}\label{eq:diver}
  (\diver \qnl, v)_{L^2(\O)} = -(\qnl,\grad v)_{L^2(\Od\times\Od)},
  \qquad \forall \qnl\in D(\diver)\subset L^2(\Od\times\Od), v\in \Udz.
\end{equation}
The domain of \(\diver\) will be denoted by \(\Qd\), that is,
\(\Qd = \{\, \qnl \in L^2(\Od\times\Od) \mid \diver \qnl \in L^2(\O)\,\}\),
which will be equipped with the graph inner product
\((\qnl,\pnl)_{\Qd} = (\qnl,\pnl)_{L^2(\Od\times\Od)} + (\diver \qnl,\diver \pnl)_{L^2(\O)}\).
We will also make use of the following closed affine subspace of \(\Qd\):
\(\Qd(f)=\{\, \qnl \in \Qd \mid \diver \qnl = f \,\}\), where we recall our blanket assumption \(f\in L^2(\O)\).

For future reference we recall the following properties of \(\diver\), see~\cite[Proposition~4.1]{eb2021}.
\begin{proposition}\label{prop:diver0}
  The following statements hold.
  \begin{enumerate}
  \item \(\Qd\) is dense in \(L^2(\Od\times\Od)\).
  \item \(\Qd\) is a Hilbert space.
  \item \(\diver: \Qd \to L^2(\O)\) is bounded and surjective.
  \item There is \(C>0\) and \(\overline{\delta}>0\), such that for each \(v \in L^2(\O)\) and each \(0 < \delta < \overline{\delta}\)
  there is \(\qnl_{v,\delta} \in \Qd\) satisfying the conditions \(\diver\qnl_{v,\delta}=v\) and
  \(\|\qnl_{v,\delta}\|_{\Qd}\leq C\|v\|_{L^2(\O)}\).
  \end{enumerate}
\end{proposition}

As we are interested in spatially heterogenous materials, we will assume that for
each pair of interacting points \((x,x') \in \Od\times\Od\),
the strength of their interaction is not only determined by \(\Wd(x-x')\),
but also by the material properties (in our case, conductivity) at these points.
This naturally leads to a concept of nonlocal conductivity denoted by \(\knl:\Od\times\Od \to \R\), where we assume that \(\knl(x,x')=\knl(x',x)\) depends on the averaged local conductivities \(\kappa(x)\) and \(\kappa(x')\).
In particular, in~\cite{andres2015type,andres2015nonlocal,andres2017convergence}
the arithmetic averaging \(2\knl(x,x')=\kappa(x)+\kappa(x')\) was assumed, and
\cite{evgrafov2019sensitivity,evgrafov2019non} considered the case of the geometric
averaging \(\knl^2(x,x')=\kappa(x)\kappa(x')\).
In this work, it will be more convenient to assume the harmonic averaging of conductivities
\begin{equation}\label{eq:harm}\knl(x,x') = 2\kappa(x)\kappa(x')[\kappa(x)+\kappa(x')]^{-1},\end{equation} which is the
same as assuming arithmetic averaging for the resistivities \(\kappa^{-1}\), see~\cite{eb2021}.
This is the assumption we make from now on.

We now turn our attention to the dual pair of variational principles, governing the nonlocal diffusion in spatially heterogeneous materials,
and the associated optimal control in the coefficients problems for conductivities.

The nonlocal Dirichlet principle is stated as wollows:
the steady-state temperature \(u\in\Udz\) is the unique minimizer of the assocated nonlocal quadratic energy functional
\begin{equation}\label{eq:def_ap}
  \begin{aligned}
  \Ip(\knl;v)&=\frac{1}{2}\ap(\knl;v,v) - \ell(v), &\qquad&\text{where}\\
  \ap(\knl;u,v)&=\int_{\Od}\int_{\Od} \knl(x,x')\grad u(x,x')\grad v(x,x')\dx\dx',%
  \end{aligned}
\end{equation}
and \(\ell\) is given by~\eqref{eq:Dirloc}, over all \(v \in \Udz\).
This unique minimizer exists for any \(0<\underline{\kappa}\le\knl\le\overline{\kappa}<+\infty\)
owing to Lax--Milgram lemma and Proposition~\ref{prop:grad0},
and can be found as the solution to the associated optimality conditions
\begin{equation}\label{eq:primalproblem}
  \ap(\knl;u,v)=\ell(v), \quad\forall v\in\Udz.
\end{equation}
Analogously to the local saddle-point problem~\eqref{eq:compliance_local_primal}, we consider the nonlocal problem:
\begin{equation}\label{eq:compliance_nl_primal}
  \max_{\kappa\in\kadmd} \min_{u\in \Udz} \Ip(\knl;u),\\
\end{equation}
where \(\kadmd =\{\, \kappa \in L^\infty(\Od) \mid   \int_{\Od} \kappa(x)\dx \leq \gamma|\Od|,
\kappa(x) \in [\underline{\kappa},\overline{\kappa}], \text{\ for almost all \(x\in\Od\)}\,\}\).

On the dual side, we can proceed as in Subsection~\ref{subsec:local} and consider a nonlocal two-point flux variable \(\qnl(x,x') = -\knl(x,x')\grad u(x,x')\in L^2(\Od\times\Od)\), and attempt to characterize it using a variational principle.
To this end, we can seek the unique pair \((\qnl,u) \in \Qd\times L^2(\O)\) satisfying the following mixed variational problem:
\begin{equation}\label{eq:mixedproblem}
  \begin{aligned}
    \ad(\knl;\qnl,\pnl) + b(\pnl,u) &= 0, &\qquad& \forall \pnl \in \Qd,\\
    b(\qnl,v) &= -\ell(v), &\qquad& \forall v \in L^2(\O),\\
  \end{aligned}
\end{equation}
where
\begin{equation}
  \begin{aligned}
    \ad(\knl;\qnl,\pnl) &= \int_{\Od}\int_{\Od} \knl^{-1}(x,x')\qnl(x,x')\pnl(x,x')\dx\dx',
    &\quad&\text{and}\\
    b(\pnl,v) &= -\int_{\O} \diver\pnl(x)v(x)\dx.
  \end{aligned}
\end{equation}
The system~\eqref{eq:mixedproblem} is well-posed, and in fact the following statement holds, see~\cite[Theorem~4.5]{eb2021}.
\begin{theorem}\label{thm:mixed_exist}
  Problem~\eqref{eq:mixedproblem} admits a unique solution \((\qnl,u)\in \Qd\times L^2(\O)\).
  This solution satisfies the stability estimate
  \begin{equation}
    \|\qnl\|_{\Qd} + \|u\|_{L^2(\O)} \le C \|f\|_{L^2(\O)},
  \end{equation}
  for some \(C>0\) independent from \(\qnl\),  \(u\), \(f\), or \(\delta \in (0,\overline{\delta})\),
  but of course dependent on \(\underline{\kappa}\) and \(\overline{\kappa}\).
\end{theorem}

There is a direct  correspondence between the unique solutions to~\eqref{eq:primalproblem} and~\eqref{eq:mixedproblem} for heat sources \(f\in L^2(\O)\).
Furthermore, the mixed variational problem~\eqref{eq:mixedproblem} is the system of optimality
conditions for the constrained quadratic optimization problem, which can be interpreted as
the nonlocal Kelvin variational principle~\cite{eb2021}:
\begin{equation}\label{eq:nlkelvin}
    \min_{\qnl\in\Qd(f)} \Id(\knl;\qnl),
\end{equation}
where \(2\Id(\knl;\qnl) = \ad(\knl;\qnl,\qnl)\).
The strong duality holds for nonlocal Dirichlet and Kelvin principles, that is, the optimal value for the problem above, which is uniquely attained owing to Theorem~\ref{thm:mixed_exist}, equals to \(-\min_{u\in \Udz} \Ip(\knl;u)\).
Consequently, the problem~\eqref{eq:compliance_nl_primal} can be equivalently stated in terms of nonlocal fluxes, leading us to the following nonlocal analogue of~\eqref{eq:compliance_local_dual}:
\begin{equation}\label{eq:compliance_nl_mixed}
  \min_{(\kappa,\qnl) \in \kadmd\times\Qd(f)} \Id(\knl;\qnl).
\end{equation}
The problem~\eqref{eq:compliance_nl_mixed} also admits optimal solutions, as may be inferred from the equivalence between~\eqref{eq:compliance_nl_primal} and~\eqref{eq:compliance_nl_mixed}, or as can be verified independently from the convexity considerations.

The following simple observation, which has been made in~\cite{eb2021}, is crucial for the developments that follow.

\begin{remark}\label{rem:asym}
  The optimal flux in~\eqref{eq:nlkelvin}, and consequently also
  in~\eqref{eq:compliance_nl_mixed}, is anti-symmetric,
  that is, \(\qnl(x,x')+\qnl(x',x) = 0\),
  for almost all \((x,x')\in \Od\times\Od\), see~\cite[Remark~4.6]{eb2021}.
  As we will encounter symmetric and anti-symmetric Lebesgue integrable functions often
  in what follows, we introduce the notation for the corresponding closed subspaces of
  \(L^p(\Od\times\Od)\):
  \[\begin{aligned}
    L^p_s(\Od\times\Od) &= \{\, \qnl \in L^p(\Od\times\Od)
  \mid \qnl(x,x')-\qnl(x',x)=0, \text{a.e.\ in \(\Od\times\Od\)}\,\}\quad\text{and}\\
  L^p_a(\Od\times\Od) &= \{\, \qnl \in L^p(\Od\times\Od)
\mid \qnl(x,x')+\qnl(x',x)=0, \text{a.e.\ in \(\Od\times\Od\)}\,\},\\
  \end{aligned}\]
  for \(p\in[1,\infty]\), and later on also for multi-indices \(p\).
  For \emph{single}-indices \(p\in[1,\infty]\) we have the identity
  \(L^p(\Od\times\Od) = L^p_s(\Od\times\Od) \oplus L^p_a(\Od\times\Od)\),
  with two subspaces orthogonal to each other when \(p=2\).
  Indeed,
  \begin{equation}\label{eq:orth}
    \begin{aligned}
    \int_{\Od}\int_{\Od} \qnl_s(x,x')\qnl_a(x,x')\dx'\dx
    &=
    -\int_{\Od}\int_{\Od} \qnl_s(x',x)\qnl_a(x',x)\dx'\dx
    \\&=-\int_{\Od}\int_{\Od} \qnl_s(x',x)\qnl_a(x',x)\dx\dx',
    \\&=-\int_{\Od}\int_{\Od} \qnl_s(x,x')\qnl_a(x,x')\dx'\dx,
    \\&\forall
    \qnl_s\in L^2_s(\Od\times\Od),
    \qnl_a\in L^2_a(\Od\times\Od),\end{aligned}
  \end{equation}
  where we use the symmetry and anti-symmetry of the functions to obtain the first equality, Fubini's theorem to obtain the second, and finally relabel the variables \(x\leftrightarrow x'\) to obtain the third.

  As a consequence of this observation, we can replace \(\Qd\) with
  its closed subspace (which is therefore a Hilbert space with respect to the induced norm) \(\Qds = \Qd\cap L^2_a(\Od\times\Od)\) in the minimization problem~\eqref{eq:nlkelvin} and the associated
  design problem~\eqref{eq:compliance_nl_mixed}, without altering the solution.
  Naturally we will also make use of the notation \(\Qds(f) = \{\, \qnl \in \Qds \mid \diver \qnl = f\,\}\).
\end{remark}

\begin{remark}\label{rem:kappa0}
  Note that one can extend the existence of solutions results for
  problems~\eqref{eq:nlkelvin} and~\eqref{eq:compliance_nl_mixed}
  even when the lower limit on conductivitiy   \(\underline{\kappa}\) defining
  \(\kadmd\) is zero, where to avoid trivial cases we assume \(\gamma>0\).
  Indeed, the function \(\R_+\times \R\ni (\kappa,q) \mapsto \kappa^{-1}q^2\)
  is convex and lower semi-continuous (with the identification \(0^{-1} 0^2 = 0\)), and the
  direct method of calculus of variations still applies.
  For the problem~\eqref{eq:nlkelvin} we get a unique solution provided there is
  a feasible two-point flux \(\qnl\in\Qd\) such that \(\Id(\knl;\qnl)<\infty\).
  In the case of~\eqref{eq:compliance_nl_mixed}, we get existence of solutions
  as \(\kappa \equiv \gamma\) is a feasible positive design, for which
  Theorem~\ref{thm:mixed_exist} provides a unique two-point flux.
  Owing to this fact the objective function is proper with respect to the feasible set
  of~\eqref{eq:compliance_nl_mixed} and the existence of solutions follows in the
  standard fashion without further assumptions.
\end{remark}

Let us conclude this subsection by summarizing the properties of the nonlocal optimal designs problems we have just introduced;
in particular see~\cite[Theorem~6.4]{eb2021}.
\begin{theorem}\label{thm:primal_summary}
  \begin{enumerate}
      \item The problems~\eqref{eq:compliance_nl_primal} and~\eqref{eq:compliance_nl_mixed} admit optimal solutions.
      \item
      \(\lim_{\delta\searrow 0} \min_{(\kappa,\qnl)\in \kadmd\times\Qd(f)}\Id(\knl;\qnl)
      = \min_{(\kappa,q)\in \kadm\times\Q(f)}\Idloc(\kappa;q)\).
      In fact, the sequence of reduced functions \(\kappa\mapsto \min_{\qnl\in \Qd(f)}\Id(\knl;\qnl)\) \(\Gamma\)-converges,
      as \(\delta\searrow 0\), towards its local limit \(\kappa\mapsto \min_{q\in \Q(f)}\Idloc(\kappa;q)\).
 \end{enumerate}
\end{theorem}

\subsection{Vanishing material fraction limit for the nonlocal problem}\label{subsec:nlvan}

Now we would like to derive a formal low volume ratio limit of~\eqref{eq:compliance_nl_mixed},
that is, a nonlocal version of~\eqref{eq:limiting_local_dual_rig}, where we proceed along the lines of
Subsection~\ref{subsec:lvan};
in particular, we assume \(\underline{\kappa}=0\)
and \(\gamma\searrow 0\).
Analogous to~\eqref{eq:kadm_gamma_loc} and~\eqref{eq:idloc}
we define
\begin{equation}\label{eq:gamma_mixed}
  \begin{aligned}
    \kadmd^\gamma &= \bigg\{\, \kappa_\gamma \in L^1(\Od) : 0 \le \kappa_\gamma(x) \le \gamma^{-1}\overline{\kappa},
    \text{\ a.e. in \(\Od\)},
    \int_{\Od} \kappa_\gamma(x)\dx \leq |\Od|\,\bigg\}, &\quad&\text{and}\\
    \id^\gamma(\qnl) &= \inf_{\kappa_\gamma \in \kadmd^\gamma} \Id(\knl_\gamma,\qnl)
    = \inf_{\kappa_\gamma \in \kadmd^\gamma}
    \frac{1}{2}\int_{\Od} \kappa_\gamma^{-1}(x)
    \int_{\Od}|\qnl(x,x')|^2\dx'\dx,
  \end{aligned}
\end{equation}
where we utilize the definition of the nonlocal conductivity \(\knl_\gamma^{-1}(x,x') = [\kappa_\gamma^{-1}(x)+\kappa_\gamma^{-1}(x')]/2\) via harmonic averaging as well as the inclusion \(\qnl \in \Qds\) to arrive at the last equality.
With these definitions~\eqref{eq:compliance_nl_mixed} is equivalent to
\begin{equation}\label{eq:igamma_nl}
  \hat{p}^*_{\delta,\gamma}=\inf_{\qnl \in \Qds(f)} \id^\gamma(\qnl).
\end{equation}
We now take a formal limit \(\gamma\searrow 0\) in~\eqref{eq:gamma_mixed}.
We are going to need Lebesgue spaces  \(L^p(\Od\times\Od)\) with mixed exponents \(p\),
see~\cite{benedek1961space} or~\cite[Section 2.48]{adams2003sobolev}.
In particular, we denote by \(L^{p,q}(\Od\times\Od)\), \(1\leq p,q \leq +\infty\) the space of
(equivalence classes of) Lebesgue measurable functions \(\qnl:\Od\times\Od \to \R\)
such that
\begin{equation}
  \|\qnl\|_{L^{p,q}(\Od\times\Od)} =
  \bigg\{\int_{\Od}\bigg[\int_{\Od} |\qnl(x,x')|^q\dx'\bigg]^{p/q}\dx\bigg\}^{1/p} < +\infty,
\end{equation}
with the usual replacement of integration with \(\esssup\) whenever \(p\) or \(q\) is infinite.
The most important for us is going to be the space \(L^{1,2}(\Od\times\Od)\), which is a separable,
non-reflexive Banach space, whose dual is isometrically isomorphic to \(L^{\infty,2}(\Od\times\Od)\)
with the usual primal-dual pairing through integration.
We would also like to remind the reader about our notation \(L_a^{p,q}(\Od\times\Od)\)
for anti-symmetric  functions in \(L^{p,q}(\Od\times\Od)\).

Similarly to Proposition~\ref{prop:id0} we have the following result.
\begin{proposition}\label{prop:id0nl}
  For each \(\qnl \in L^{1,2}_a(\Od\times\Od)\supset \Qds\) the formal limiting optimization
  problem
  \begin{equation}
    \min_{\kappa_0 \in \kadmd^0}\Id(\knl_0;q),
  \end{equation}
  where \(2\knl^{-1}_0(x,x') = \kappa^{-1}_0(x)+\kappa^{-1}_0(x')\),
  admits an optimal solution defined pointwise by
  \begin{equation}
    \begin{aligned}
      \hat{\kappa}_0(x) = \hat{c}_\delta
      \bigg[\int_{\Od}|\qnl(x,x')|^2\dx'\bigg]^{1/2},\quad\text{where\ }
      \hat{c}_\delta = |\Od|\|\qnl\|_{L^{1,2}(\Od\times\Od)}^{-1}.
    \end{aligned}
  \end{equation}
\end{proposition}
\begin{proof}
  In view of the last equality in~\eqref{eq:gamma_mixed} it is sufficient to appeal
  to Proposition~\ref{prop:id0} and identify \(|q(x)| = [\int_{\Od}|\qnl(x,x')|^2\dx']^{1/2}\).\footnote{%
  Strictly speaking, we also need to identify \(\O=\Od\), which is a great abuse of
  the previously introduced notation.}
\end{proof}
As a direct consequence of Proposition~\ref{prop:id0nl} we have the equality
\(\id^0(\qnl) = (2|\Od|)^{-1}\|\qnl\|^2_{L^{1,2}(\Od\times\Od)}\).
To complete the derivation of the formal limiting problem we need to extend
the non-local divergence operator to a broader class of functions
\(L^{1,2}(\Od\times\Od)\supset L^2(\Od\times\Od)\).
In what follows we will use the notation \(W^{1,\infty}_0(\O)\) to denote the
functions in \(W^{1,\infty}(\O)\), whose extension by \(0\) in \(\Rn\setminus\O\)
is in \(W^{1,\infty}(\Rn)\).

\begin{proposition}\label{prop:grad}
  For each \(u \in W^{1,\infty}_0(\O)\) and a measurable
  set \(E \subseteq \Od\) we have the estimate
  \begin{equation}\label{eq:grad}
    \esssup_{x\in\Od}\bigg[\int_{E}|\grad u(x,x')|^2\dx'\bigg]^{1/2}
    \leq
     \|u\|_{W^{1,\infty}(\O)}
    \sup_{F:F\subseteq \Rn, |F|\leq |E|}
    \bigg[\int_{F} \Wd^2(z)|z|^2\,{\mathrm d}z\bigg]^{1/2}.
  \end{equation}
  In particular, \(\grad u \in L^{\infty,2}_a(\Od\times\Od)\),
  and \(\|\grad u\|_{L^{\infty,2}(\Od\times\Od)}\leq K_{2,n}^{-1/2}\|u\|_{W^{1,\infty}(\O)}\).
\end{proposition}
\begin{proof}
  Let us chose a Lipschitz continuous representative of \(u\in W^{1,\infty}_0(\O)\) and
  extend it by zero outside of \(\O\).
  Then, \(\|u\|_{W^{1,\infty}(\Rn)} = \|u\|_{W^{1,\infty}(\O)}\), and
  \(|u(x)-u(x')|\leq |x-x'|\|\nabla u\|_{L^\infty(\Rn;\Rn)}\), for all \((x,x')\in \Od\times\Od\),
  see for example~\cite[Proposition~9.3]{brezis2010functional}.
  Then the claims follow immediately from the inequality
  \begin{equation*}
    |\grad u(x,x')| \leq  \|\nabla u\|_{L^{\infty}(\Rn;\Rn)} \Wd(x-x')|x-x'|, \qquad\forall (x,x')\in \Od\times\Od,
  \end{equation*}
  and the normalization condition \(\int_{\Rn}\Wd^2(z)|z|^2\,\mathrm{d}z = K_{2,n}^{-1}\).
\end{proof}

\begin{proposition}\label{prop:closable}
  Unbounded linear operator \(\diver: \Qd\subset L^{1,2}(\Od\times\Od)\to L^2(\O)\) is closable.
  Its closure \(\diverx: D(\diverx)\subset L^{1,2}(\Od\times\Od)\to L^2(\O)\) satisfies the
  equality
  \begin{equation}\label{eq:divx}
    (\diverx \qnl,v)_{L^2(\O)}=-\int_{\Od}\int_{\Od} \qnl(x,x')\grad v(x,x')\dx'\dx,
    \qquad \forall v\in W^{1,\infty}_0(\O).
  \end{equation}
\end{proposition}
\begin{proof}
  Suppose that a sequence \(\qnl_k \in \Qd=D(\diver)\subset L^2(\Od\times\Od)\),
  \(k=1,2,\dots\) converges to
  zero in \(L^{1,2}(\Od\times\Od)\), and that \(u_k=\diver\qnl_k\) converges
  to \(\hat{u}\in L^2(\O)\).
  Then, for each \(v\in W^{1,\infty}_0(\O)\subset\Udz\) we have
  \[\begin{aligned}
  |(\hat{u},v)_{L^2(\O)}| &= \lim_{k\to\infty} |(u_k,v)_{L^2(\O)}|
  =\lim_{k\to\infty} \bigg|\int_{\Od}\int_{\Od}\qnl_k(x,x')\grad v(x,x')\dx'\dx\bigg|
  \\&
  \leq \|\grad v\|_{L^{\infty,2}(\Od\times\Od)}\lim_{k\to\infty}\|\qnl_k\|_{L^{1,2}(\Od\times\Od)}
  =0,
  \end{aligned}\]
  in view of Proposition~\ref{prop:grad} and H{\"o}lder's inequality.
  Since \(W^{1,\infty}_0(\O)\) is dense in \(L^2(\O)\) it follows that \(\hat{u}=0\).
\end{proof}

Before we proceed further let us introduce the following notation.
Let \(\Qdx = \{\, \qnl \in L^{1,2}(\Od\times\Od) \mid \diverx \qnl \in L^2(\O)\,\}\), which in view of Proposition~\ref{prop:closable}
is a Banach space when equipped with the graph norm
\(\|\qnl\|_{\Qdx}^2 = \|\qnl\|^2_{L^{1,2}(\Od\times\Od)} + \|\diverx \qnl\|^2_{L^2(\O)}\).
Let further \(\Qdx^a = \Qdx \cap L^{1,2}_a(\Od\times\Od)\) be its closed subspace,
which is therefore also a Banach space with the induced norm.
Finally, for \(f\in L^2(\O)\) we define the closed affine subspaces
\(\Qdx(f) = \{\, \qnl \in \Qdx \mid \diverx \qnl = f \,\}\)
and
\(\Qdx^a(f) = \{\, \qnl \in \Qdx^a \mid \diverx \qnl = f \,\}\).
Note that \(\Qdx(f)\), and in fact also \(\Qdx^a(f)\), are nonempty owing to Proposition~\ref{prop:diver0}.4 and H{\"o}lder's inequality;
indeed the flux \(\qnl_{v,\delta}\) of Proposition~\ref{prop:diver0}.4
is in fact in \(L^{2}_a(\Od\times\Od)\subset L^{1,2}_a(\Od\times\Od)\), see~\cite[Proposition~4.1]{eb2021}.

After this preliminary work, we can define the following formal limiting problem for~\eqref{eq:compliance_nl_mixed}:
\begin{equation}\label{eq:prob_nl}
  p^*_{\delta,a} = \inf_{\qnl \in \Qdx^a(f)} \|\qnl\|_{L^{1,2}(\Od\times\Od)} =
  \inf_{\qnl \in \Qdx^a} [\|\qnl\|_{L^{1,2}(\Od\times\Od)} + \ind_{\{f\}}(\diverx \qnl)].
\end{equation}
Note that this problem should be viewed as a nonlocal analogue of~\eqref{eq:limiting_local_dual_rig}.

\section{Analysis of the limiting local problem}\label{sec:analysis_loc}
In this section we add some rigour to the formal derivation of~\eqref{eq:limiting_local_dual_rig} from~\eqref{eq:compl_flux_loc}.
We reiterate that the problem~\eqref{eq:limiting_local_dual_rig} and the aforementioned limiting process are very well studied in the existing literature on the subject; we refer the interested reader to for example~\cite{bouchitte2008michell,bouchitte2001characterization,babadjian2021shape,Olbermann2017,bolbotowski2022optimal}.
Nevertheless, the material in this section can be treated as a preview of the forthcoming derivations for the non-local problem~\eqref{eq:prob_nl}, and is therefore included to keep this work self-contained.
To this end we introduce the \(L^2(\O)\)-Tikhonov regularization of~\eqref{eq:limiting_local_dual}.
Let \(\beta>0\) be a regularization parameter and \(j_\beta(q) = \|q\|_{L^1(\O;\Rn)} + \frac{\beta}{2}\|q\|^2_{L^2(\O;\Rn)}\).
We consider the pair of primal
\begin{align}\label{eq:Tikhonov_local}
  p^*_{\text{loc},\beta} &= \inf_{q\in \Q} [ j_\beta(q) + \ind_{\{f\}}(\Div q)],
  \intertext{and dual}
    \label{eq:Tikhonov_dual_local}
  -d^*_{\text{loc},\beta} &= \inf_{u\in L^2(\O)} [ j_\beta^*(\Div^* u) -\ell(u)],
\end{align}
problems,  where \(\Div^* \in \mathcal{B}(L^2(\O), \Q')\) is the negative adjoint of \(\Div\in \mathcal{B}(\Q,L^2(\O))\),
and \(j_\beta^*: \Q'\to\R\cup\{+\infty\}\) is the convex conjugate of \(j_\beta\) given by
\begin{equation}\label{eq:jstar_loc}\begin{aligned}
  j_\beta^*(p) = \sup_{q\in\Q}[\langle p,q\rangle - j_\beta(q)].
\end{aligned}\end{equation}

Our strategy now is as follows. The  family of problems~\eqref{eq:Tikhonov_local} provides us with the approximation of \(p^*_{\text{loc}}\) from above.
We will see further, that they also estimate \(\hat{p}^*_{\text{loc},\gamma}\) from above.
To obtain lower estimates, we would like to use~\eqref{eq:Tikhonov_dual_local} to approximate \(pd^*_{\text{loc}}\) from below. However, in order to do this, we need to extend the feasible set of~\eqref{eq:limiting_local_predual} to include potential limit points of~\eqref{eq:Tikhonov_dual_local}.
Throughout this section \(\overline{\beta}\) is a fixed arbitrary positive constant.
\begin{proposition}\label{prop:pdeq}
  Consider the following relaxation of~\eqref{eq:limiting_local_predual}:
  \begin{equation}\label{eq:limiting_local_predual_relax}
    -\overline{pd}^*_{\text{loc}} = \inf_{u \in W^{1,\infty}_0(\O)} [\ind_{B_{L^\infty(\O;\Rn)}}(\nabla u)-\ell(u)],
  \end{equation}
  in which we replace the space of continuously differentiable functions \(C^1_0(\cl\O)\) with
  a larger Sobolev space \(W^{1,\infty}_0(\O)\).
  Then
  \(\overline{pd}^*_{\text{loc}}=pd^*_{\text{loc}}\).
\end{proposition}
\begin{proof}
  The inequality \(\overline{pd}^*_{\text{loc}}\geq pd^*_{\text{loc}}\) follows
  immediately from the inclusions \(C^1_0(\cl\O) \subset W^{1,\infty}_0(\O)\)
  and \(B_{C^0(\cl\O;\Rn)}\subset B_{L^\infty(\O;\Rn)}\).
  To show the opposite inequality, let \(u \in W^{1,\infty}_0(\O;\Rn)\) be such that
  \(\|\nabla u\|_{L^\infty(\O;\Rn)} \leq 1\).
  For an arbitrary \(\varepsilon >0\), according to \cite[Remark~4.1]{deville2019approximation},
  \(\exists v_\epsilon \in C^\infty_c(\O)\)
  such that the Lipschitz constant for \(v_\epsilon\) on \(\cl\O\) is no greater
  than the Lipschitz constant for \(u\) and \(\|u - v_{\epsilon}\|_{C^0(\cl\O)} \leq \varepsilon\).
  Therefore, \(\overline{pd}^*_{\text{loc}} \leq pd^*_{\text{loc}} + \varepsilon \|f\|_{L^1(\O)}\),
  thereby concluding the proof.
\end{proof}

We begin by describing the behaviour of~\eqref{eq:Tikhonov_dual_local} as \(\beta\searrow 0\).

\begin{proposition}\label{prop:tikh_exist_loc}
 The following statements hold.
  \begin{enumerate}
    \item For each \(\beta>0\) the infimum in~\eqref{eq:Tikhonov_local} is uniquely attained.
    \item For each \(\beta>0\), the strong duality holds: \(p^*_{\text{loc},\beta} = d^*_{\text{loc},\beta}\),
    and the infimum in~\eqref{eq:Tikhonov_dual_local} is attained.
    \item For each \(\overline{\beta}>0\), the family \(j^*_{\beta}(\Div^*\cdot)\), \(\beta \in (0,\overline{\beta}]\) is equi-coercive on \(L^2(\O)\).
    \item
    The functionals \(j^*_\beta(\Div^*\cdot):
    L^2(\O)\to\R\cup\{+\infty\}\) \(\Gamma\)-converge to
    \(\hat{\jmath}: L^2(\O)\to\R\cup\{+\infty\}\) as \(\beta\searrow 0\), where
      \[
      \hat{\jmath}(u) = \begin{cases}
      \ind_{B_{L^{\infty}(\O;\Rn)}}(\nabla u), &\quad u\in W^{1,2}_0(\O),\\
      +\infty,&\quad\text{otherwise}.
      \end{cases}.\]
   \item
  We have the equality \(\lim_{\beta\searrow 0} d^*_{\text{loc},\beta} = \overline{pd}^*_{\text{loc}}\), and
  any \(L^2(\O)\)-limit point of solutions to~\eqref{eq:Tikhonov_dual_local} is a solution
  to~\eqref{eq:limiting_local_predual_relax}.
 \end{enumerate}
\end{proposition}
\begin{proof}
\begin{enumerate}
\item   The problem~\eqref{eq:Tikhonov_local} amounts to minimizing a convex, lower semicontinuous, and coercive functional \(j_\beta(q) + \ind_{\{f\}}(\Div q)\) over the Hilbert space \(\Q\).
    The solution exists owing to the generalized Weierstrass' theorem, see for example~\cite[Corollary 3.23]{brezis2010functional}.
%
\item Note that the conditions for strong Fenchel duality (see, for example, \cite[Theorem~4.4.3]{borwein2004techniques}) are fulfilled:
\(j_\beta: \Q\to\R\) and \(\ind_{\{f\}}: L^2(\O)\to\R\cup\{+\infty\}\) are both convex and lower semi-continuous, and \(\Div: \Q\to L^2(\O)\) is bounded and surjective.
Therefore the infimum in~\eqref{eq:Tikhonov_dual_local} is attained when finite.
Its finiteness follows from the feasbility of \(0\) in~\eqref{eq:Tikhonov_dual_local}, and the finiteness of \(p^*_{\text{loc},\beta}\) established previously.
\item
Note that owing to H{\"o}lder's inequality we have the estimate
\[
j_\beta(q) \leq \frac{\beta}{2} \|q\|^2_{L^2(\O;\Rn)} + |\O|^{1/2} \|q\|_{L^2(\O;\Rn)}, \qquad \forall q \in L^2(\O;\Rn).
\]
Let us now take an arbitrary \(u\in L^2(\O)\) and define
\[
\alpha = \sup_{q\in \Q, \|q\|_{L^2(\O;\Rn)}\leq 1} \langle \Div^* u, q\rangle
       = \sup_{q\in \Q, \|q\|_{L^2(\O;\Rn)}\leq 1} (u, \Div q)_{L^2(\O)}.
\]
If \(\alpha = +\infty\) then also
\begin{equation}\label{eq:jstar_lb}\begin{aligned}
   j_\beta^*(\Div^* u) &= \sup_{q\in\Q}[\langle \Div^* u,q\rangle - j_\beta(q)]
   \\&\geq
   \sup_{q\in \Q, \|q\|_{L^2(\O;\Rn)}\leq 1}\bigg[\langle \Div^* u,q\rangle - \frac{\beta}{2}-|\O|^{1/2}\bigg] = +\infty.
\end{aligned}\end{equation}
If, on the other hand, \((u, \Div q)_{L^2(\O)} \leq \alpha \|q\|_{L^2(\O;\Rn)}\) for all \(q\in C^1_c(\Rn)\subset \Q\) with \(\alpha < +\infty\), then \(u\in W^{1,2}_0(\O)\), see~\cite[Proposition~9.18]{brezis2010functional} and~\cite[Theorem~2.4]{girault_raviart}.
Consequently, owing to the dense inclusion \(\Q\in L^2(\O;\Rn)\)
and \(L^2(\O)\)-continuity of \(j_\beta\) we can write
\[
   j_\beta^*(\Div^* u) = \sup_{q\in\Q}[\langle \Div^* u, q\rangle - j_\beta(q)]
   = \sup_{q \in L^2(\O;\Rn)}[(\nabla u,q)_{L^2(\O;\Rn)} - j_\beta(q)].
\]
The second supremum above is attained at
\[
 \hat{q}_\beta = \argmin_{q\in L^2(\O;\Rn)}[
 \|q\|_{L^1(\O;\Rn)} + \frac{\beta}{2}\|q-\beta^{-1}\nabla u\|^2_{L^2(\O;\Rn)}
 ],
\]
which is given  pointwise by the soft thresholding operator
\(\hat{q}_{\beta}(x)=\beta^{-1}\max\{1-|\nabla u(x)|^{-1},0\}{\nabla u}(x)\).
Substituting this back into~\eqref{eq:jstar_loc} we arrive at the following expression:
\begin{equation}\label{eq:jstar_explicit}
  \begin{aligned}
j^*_\beta(\Div^* u)&=
\frac{1}{2\beta}\|\max\{|\nabla u(\cdot)|-1,0\}\|^2_{L^2(\O;\Rn)} \\&\leq \frac{1}{2\beta}\|\nabla u\|^2_{L^2(\O)} < +\infty.
\end{aligned}
\end{equation}
To summarise, if \(u\in L^2(\O)\setminus W^{1,2}_0(\O)\) then
\(j^*_\beta(\Div^* u)\equiv+\infty\).
If, on the other hand, \(u_k \in W^{1,2}_0(\O)\) is such that \(\|u_k\|_{L^2(\O)}\to\infty\), then owing to Poincar{\'e}'s inequality (see for example~\cite[Corollary~9.19]{brezis2010functional}) we have \(\|\nabla u_k\|_{L^2(\O;\Rn)} \geq C(\O)\|u_k\|_{L^2(\O)}\), for some \(C(\O)>0\) and therefore also
\((2\beta)^{-1}\|\max\{|\nabla u_k(\cdot)|-1,0\}\|^2_{L^2(\O;\Rn)} \geq (2\overline{\beta})^{-1}\|\max\{|\nabla u_k(\cdot)|-1,0\}\|^2_{L^2(\O;\Rn)}\to\infty\),
uniformly with respect to \(\beta \in (0,\overline{\beta}]\).
\item
The family of functionals \(j_\beta(\cdot)\) is monotonically non-decreasing with respect to \(\beta\), and therefore owing to the order-reversing property of convex conjugates the family of functionals \(j^*_\beta(\cdot)\) is monotonically non-increasing with respect to \(\beta\).
Furthremore, convex conjugates are lower semicontinuous, and \(\Div^*\) is an adjoint of a linear bounded operator, consequently the composite functionals \(j^*_\beta(\Div^* \cdot)\) are lower semicontinuous.
Therefore, in view of~\cite[Remark 2.12]{braides2006handbook} the \(\Gamma\)-limit of \(j^*_\beta(\Div^* \cdot)\) is given by the pointwise limit of these functionals.
This pointwise limit, for each \(u\in L^2(\O)\), is given precisely by \(\hat{\jmath}(u)=\lim_{\beta\searrow 0} j^*_\beta(\Div^* u)\),
see the bound~\eqref{eq:jstar_lb} and the explicit expression~\eqref{eq:jstar_explicit} for \(j^*_\beta\).
\item
Note that \(j^*_\beta(\Div^*\cdot)-\ell\) is still equi-coercive on \(L^2(\O)\) and \(\Gamma\)-converges to \(\hat{\jmath}(\cdot) - \ell\), since \(\Gamma\)-convergence is stable with respect to continuous perturbations, see for example~\cite[Remark 2.2]{braides2006handbook}.

We would like to show that the problem~\eqref{eq:limiting_local_predual_relax} is equivalent to minimizing \(\hat{\jmath}(\cdot) - \ell\) over \(L^2(\O)\),
for which it is sufficient to show that \(\hat{\jmath}(u)<+\infty\) implies \(u\in W^{1,\infty}_0(\O)\).
Suppose that \(u\in L^2(\O)\) is such that \(\hat{\jmath}(u)<+\infty\);
in particular \(u \in W^{1,2}_0(\O)\) and \(\|\nabla u\|_{L^\infty(\O;\Rn)}\leq 1\).
Therefore, \(u\) extended by zero on \(\Rn\setminus \O\) is in \(W^{1,2}(\Rn)\) with \(\|\nabla u\|_{L^\infty(\Rn; \Rn)}\le 1\).
Owing to Gagliardo--Nirenberg--Sobolev inequality, \(u \in L^p(\Rn)\) for some \(p>n\), and consequently also is H{\"o}lder continous on \(\Rn\) owing to Morrey's inequality.
In particular, we have that \(u \in W^{1,\infty}(\Rn)\) and consequently
also \(u \in W^{1,\infty}_0(\O)\).

 Therefore, in view of point 4, it remains to appeal to the fundamental theorem of \(\Gamma\)-convergence, see for example~\cite[Theorem 2.10]{braides2006handbook}.
\end{enumerate}
\end{proof}

The convergence of optimal fluxes follows from these preliminary results.
\begin{proposition}\label{thm:tikhonov_local}
  Let \(\hat{q}_\beta \in \Q\), \(\beta \in (0,\overline{\beta}]\),
  be the family of optimal solutions to~\eqref{eq:Tikhonov_local},
  which exist according to Proposition~\ref{prop:tikh_exist_loc} point 1.
  The following statements hold:
  \begin{enumerate}
    \item
    \(p^*_{\text{loc}} = \lim_{\beta\searrow 0} p^*_{\text{loc},\beta}\).
    In fact, \(p^*_{\text{loc}} = \lim_{\beta\searrow 0}
    \|\hat{q}_\beta\|_{L^1(\O;\Rn)}\) and
    \(\lim_{\beta\searrow 0} (\beta/2) \|\hat{q}_\beta\|^2_{L^2(\O;\Rn)} = 0\).
    \item The family of vector-valued Radon measures \(\hat{q}_\beta\mathcal{H}^n\)
    is relatively sequentially compact with respect to the weak convergence of measures in \(\mathcal{M}(\cl\O;\Rn)\),
    and each limit point of this family as \(\beta\searrow 0\),
    is an optimal solution to~\eqref{eq:limiting_local_dual_rig}.
  \end{enumerate}
\end{proposition}
\begin{proof}
  \begin{enumerate}
  \item
  Since each \(\hat{q}_\beta\) is feasible in~\eqref{eq:limiting_local_dual_rig}, we have an estimate
  \begin{equation}\label{eq:prop90}\begin{aligned}
  p^*_{\text{loc}} &\leq
  \liminf_{\beta\searrow 0} \|\hat{q}_\beta\|_{L^1(\O;\Rn)}
  \leq
  \limsup_{\beta\searrow 0} \|\hat{q}_\beta\|_{L^1(\O;\Rn)}
  \\&\leq
  \lim_{\beta\searrow 0}\left[
  \|\hat{q}_\beta\|_{L^1(\O;\Rn)} + \frac{\beta}{2}\|\hat{q}_\beta\|_{L^2(\O;\Rn)}^2\right]
  = \lim_{\beta\searrow 0}p^*_{\text{loc},\beta}
  = p^*_{\text{loc}},
  \end{aligned}\end{equation}
  where the last equality follows from Proposition~\ref{prop:tikh_exist_loc}, points 2 and 5, Proposition~\ref{prop:pdeq}, and Theorem~\ref{thm1}.
  Therefore, equalities must hold throughout, and the claim follows.
  \item
  Note that~\eqref{eq:prop90} implies that
  the norms \(\|\hat{q}_\beta\|_{L^1(\O;\Rn)}\) remain bounded as
  \(\beta\searrow 0\).
  Consequently, the family \(\hat{q}_\beta\mathcal{H}^n\) is relatively sequentially compact with respect to the weak convergence of measures in \(\mathcal{M}(\cl\O;\Rn)\).

  Let \(\hat{q}\in \mathcal{M}(\cl\O;\Rn)\) be an arbitrary limit point of \(\hat{q}_\beta\mathcal{H}^n\)  along some subsequence \(\beta_{k}\searrow 0\).
  Then for each \(\phi \in C^1_0(\cl\O)\subset W^{1,2}_0(\O)\) we have the equality
  \[\begin{aligned}
  -\int_{\cl \O}\nabla \phi(x)\cdot \mathrm{d}\hat{q}(x)
  &=
  -\lim_{\beta_{k}\searrow 0} \int_{\O}\nabla \phi(x)\cdot \hat{q}_{\beta_{k}}(x) \dx
  \\&= \lim_{\beta_{k}\searrow 0} \int_{\O} \phi(x)\Div \hat{q}_{\beta_{k}}(x) \dx = \ell(\phi),
  \end{aligned}\]
  thereby establishing the feasibility of \(\hat{q}\) in~\eqref{eq:limiting_local_dual_rig}.
  Consequently, we have the bounds
  \[\begin{aligned}p^*_{\text{loc}} &\leq |\hat{q}|(\cl\O) \leq
  \liminf_{\beta_{k}\searrow 0} \| \hat{q}_{\beta_{k}} \|_{L^1(\O;\Rn)}
  =p^*_{\text{loc}},\end{aligned}\]
  thereby showing the optimality of \(\hat{q}\) in~\eqref{eq:limiting_local_dual_rig}.
  \end{enumerate}
\end{proof}

We now turn our attention to the problems~\eqref{eq:compl_flux_loc},
for which we can state the convergence result effectively mirroring Proposition~\ref{thm:tikhonov_local}.

\begin{theorem}\label{thm:vanish_limit_local}
  The family of optimal solutions \(q_\gamma \in \Q\)
  to~\eqref{eq:compl_flux_loc}, is bounded in \(L^1(\O;\Rn)\)
  as \(\gamma\searrow 0\).
  The corresponding family of vector valued Radon measures
  \(q_\gamma\mathcal{H}^n\),
  is relatively sequentially compact with respect to the weak convergence in \(\mathcal{M}(\cl\O;\Rn)\).
  Each limit point of this sequence is a solution to~\eqref{eq:limiting_local_dual_rig}.
\end{theorem}
\begin{proof}
  We will prove the following equality:
  \begin{equation}\label{eq:thm40}
    \lim_{\gamma\searrow 0}\hat{p}^*_{\text{loc},\gamma}
    = \frac{1}{2|\O|} [p^*_{\text{loc}}]^2,
  \end{equation}
  where the expression in the right hand side comes from the relation between \(i^0_{\text{loc}}(\cdot)\) and \(\|\cdot\|_{L^1(\O;\Rn)}\) established in Proposition~\ref{prop:id0} and the discussion immediately thereafter.

  For each \(\gamma>0\) we have the inclusion \(\kadm^\gamma \subseteq\kadm^0\), which in turn implies the inequality
  \(i^0_{\text{loc}}(q) \leq i^\gamma_{\text{loc}}(q)\),
  \(\forall q \in \Q\).
  In particular, in view of Proposition~\ref{prop:id0}, and noting that each \(q_\gamma\) is feasible for~\eqref{eq:limiting_local_dual_rig}, we have the estimates
  \begin{equation}\label{eq:thm41}
    \frac{1}{2|\O|} [p^*_{\text{loc}}]^2
    \leq
    \frac{1}{2|\O|}\|q_\gamma\|^{2}_{L^1(\O;\Rn)}
    = i^0_{\text{loc}}(q_\gamma)
    \leq i^\gamma_{\text{loc}}(q_\gamma)
    = \hat{p}^*_{\text{loc},\gamma}.
  \end{equation}
We will now estimate \(\hat{p}^*_{\text{loc},\gamma}\) from above using Tikhonov regularization problems~\eqref{eq:Tikhonov_local}.
Let us put \([\kadm^\gamma]^{-1} = \{\, \kappa^{-1} \mid \kappa \in \kadm^\gamma\,\}\),
\(\gamma \geq 0\).
Then \(\gamma\overline{\kappa}^{-1} + [\kadm^0]^{-1} \subseteq [\kadm^\gamma]^{-1}\), \(\forall \gamma\geq 0\).
Therefore,
\begin{equation}\label{eq:thm42}
\begin{aligned}
  \hat{p}^*_{\text{loc},\gamma}
  &= \inf_{q\in\Q(f)}
  \inf_{\kappa \in \kadm^\gamma} \frac{1}{2}\int_{\O} \kappa^{-1}(x) |q(x)|^2\dx
  \\
  &\leq \inf_{q\in\Q(f)}
  \inf_{\kappa \in \kadm^0} \frac{1}{2}\int_{\O} [\gamma\overline{\kappa}^{-1} + \kappa^{-1}(x)] |q(x)|^2\dx
  \\&= \inf_{q\in\Q(f)}\bigg[\idloc^0(q) + \frac{\gamma\overline{\kappa}^{-1}}{2}\|q\|_{L^2(\O;\Rn)}^2\bigg]
  \\&\leq
  \idloc^0(\hat{q}_{\gamma\overline{\kappa}^{-1}}) + \frac{\gamma\overline{\kappa}^{-1}}{2}\|\hat{q}_{\gamma\overline{\kappa}^{-1}}\|_{L^2(\O;\Rn)}^2
  \\&= \frac{1}{2|\O|}\|\hat{q}_{\gamma\overline{\kappa}^{-1}}\|^{2}_{L^1(\O;\Rn)}
  + \frac{\gamma\overline{\kappa}^{-1}}{2}\|\hat{q}_{\gamma\overline{\kappa}^{-1}}\|_{L^2(\O;\Rn)}^2,
\end{aligned}
\end{equation}
where \(\hat{q}_{\gamma\overline{\kappa}^{-1}} \in \Q(f)\) is the unique solution to~\eqref{eq:Tikhonov_local}
for \(\beta=\gamma\overline{\kappa}^{-1}>0\).
It remains to let \(\gamma\searrow 0\) and to utilize
Proposition~\ref{thm:tikhonov_local}, point 1, to arrive to~\eqref{eq:thm40} from~\eqref{eq:thm41} and~\eqref{eq:thm42}.

To conclude the proof it is sufficient to argue along the lines of
Proposition~\ref{thm:tikhonov_local}, point 2.
Indeed, \eqref{eq:thm41} and~\eqref{eq:thm42} imply that
\(\lim_{\gamma\searrow 0} \idloc^0(q_{\gamma}) =
\lim_{\gamma\searrow 0} (2|\O|)^{-1} \|q_{\gamma}\|_{L^1(\O;\Rn)}^2
=
(2|\O|)^{-1} [p^*_{\text{loc}}]^2\),
and ultimately the existence, feasibility, and optimality of limit points
of \(q_{\gamma}\mathcal{H}^n\) in~\eqref{eq:limiting_local_dual_rig}.
\end{proof}

\section{Analysis of the limiting nonlocal problem}\label{sec:analysis}

\subsection{Existence of solutions to~\eqref{eq:prob_nl}}\label{sec:existence}

Unfortunately, as we mentioned previously, the natural space \(L^{1,2}(\Od\times\Od)\) for~\eqref{eq:prob_nl}
is not reflexive~\cite[Theorem~1, p.~306]{benedek1961space},
thereby seemingly presenting us with the same difficulties as problem~\eqref{eq:limiting_local_dual}.
However, its closed subspace \(L^{1,2}_a(\Od\times\Od)\) is significantly smaller than
\(L^{1,2}(\Od\times\Od)\); indeed, for a function \(\qnl \in L^{1,2}(\Od\times\Od)\)
there is little reason for its anti-symmetric and symmetric parts to be in \(L^{1,2}(\Od\times\Od)\).
In this section we will establish that the infimum \(p^*_{\delta,a}\) in~\eqref{eq:prob_nl} is attained.

Let us first recall the definition of biting convergence.
\begin{definition}\label{def:bite}
  Let \(X\) be a Banach space.
  A sequence \(q_k \in L^1(\Od;X)\), \(k=1,2,\dots\) is said to converge to
  \(\hat{q} \in L^1(\Od,X)\) in the biting sense (denoted by
  \(q_k\bite\hat{q}\)) if there is a non-increasing sequence of sets
  \(\Od\supseteq E_1 \supseteq E_2 \supseteq \dots\)
  with \(\lim_{m\to\infty}|E_m|=0\) such that \(q_k\rightharpoonup \hat{q}\),
  weakly in \(L^1(\Od\setminus E_m;X)\), for each \(m=1,2,\dots\)
\end{definition}
Reflexive Banach space version of the famous Chacon's biting lemma,
see~\cite{ball1989remarks}, asserts that any bounded set in \(L^1(\O;X)\)
with \(X\) reflexive
is sequentially relatively compact with respect to the biting convergence.
From now on we will utilize this property for functions in \(L^{1,2}_a(\Od\times\Od)\), which owing
to Fubini's theorem we will isometrically identify as the elements of \(L^1(\Od;L^2(\Od))\).
It turns out that the anti-symmetry of the functions in \(L^{1,2}_a(\Od\times\Od)\) can be utilized
to strengthen the biting convergence.
In particular, we have the following type of compact embedding of \(L^{1,2}_a(\Od\times\Od)\) into \(L^1(\Od\times\Od)\).
\begin{proposition}\label{prop:compact}
  The unit ball \(B_{L^{1,2}_a(\Od\times\Od)}\) is sequentially compact with respect to the biting convergence
  in \(L^1(\Od;L^2(\Od))\).
  Additionally, such biting convergence of sequences in \(B_{L^{1,2}_a(\Od\times\Od)}\) implies the weak
  convergence in \(L^1(\Od\times\Od)\), and therefore \(B_{L^{1,2}_a(\Od\times\Od)}\) is weakly sequentially
  compact in \(L^1(\Od\times\Od)\).
\end{proposition}
\begin{proof}
  Let \(\{\qnl_k\},\) \(k=1,2,\dots\) be an arbitrary sequence in \(B_{L^{1,2}_a(\Od\times\Od)}\).
  Owing to the reflexive Banach space version of Chacon's biting lemma,
  see~\cite{ball1989remarks}, there is a function
  \(\hat{\qnl}\in L^1(\O;L^2(\O))\) and a subsequence \(k'\) of \(\N\)
  such that \(\qnl_{k'} \bite \hat{\qnl}\).
  Let us now take an arbitrary \(\pnl \in L^\infty(\Od\times\Od)\subset
  L^\infty(\Od;L^2(\Od))\).
  Then
  \[
  \begin{aligned}
  &\bigg|\int_{\Od}\int_{\Od} [\qnl_{k'}(x,x')-\hat{\qnl}(x,x')]\pnl(x,x')\dx'\dx\bigg|
  \\&\le
  \bigg|\int_{\Od\setminus E_m}\int_{\Od} [\qnl_{k'}(x,x')-\hat{\qnl}(x,x')]\pnl(x,x')\dx'\dx\bigg|
  \\&+
  \int_{E_m}\int_{\Od} |\qnl_{k'}(x,x')\pnl(x,x')|\dx'\dx
  \\&+
  \int_{E_m}\int_{\Od} |\hat{\qnl}(x,x')\pnl(x,x')|\dx'\dx
  = I_{1,m,k'} + I_{2,m,k'} + I_{3,m},
  \end{aligned}
  \]
  where \(\{E_m\}\) is vanishing non-increasing sequence of measurable
  subsets of \(\Od\) as discussed in Definition~\ref{def:bite}.
  Using the anti-symmetry of \(\qnl_{k'}\), Fubini's theorem, and H{\"o}lder's inequality,
  we estimate \(I_{2,m,k'}\) as
  \[
  \begin{aligned}
  I_{2,m,k'} &\le
  \|\pnl\|_{L^\infty(\Od\times\Od)}
  \int_{E_m}\int_{\Od} |\qnl_{k'}(x,x')|\dx'\dx
  \\&=
  \|\pnl\|_{L^\infty(\Od\times\Od)}
  \int_{\Od}\int_{E_m} |\qnl_{k'}(x,x')|\dx'\dx
  \\&\le
  \|\pnl\|_{L^\infty(\Od\times\Od)}
  \int_{\Od}
  \bigg[\int_{\Od} |\qnl_{k'}(x,x')|^2\dx'\bigg]^{1/2}
  \bigg[\int_{\Od} |\chi_{E_m}(x')|^2\dx'\bigg]^{1/2}
  \dx
  \\&=
  \|\pnl\|_{L^\infty(\Od\times\Od)}
  \|\qnl_{k'}\|_{L^{1,2}(\Od\times\Od)} |E_m|^{1/2},
  \end{aligned}
  \]
  where \(\chi_{E_m}\) is the characteristic function of \(E_m\).
  As \(\|\qnl_k\|_{L^{1,2}(\Od\times\Od)}\le 1\), for \(k=1,2,\dots\), for an arbitrary but fixed
  \(\epsilon >0\) we can choose \(M\in \N\) so large that \(I_{2,m,k'}<\epsilon\),
  for all \(m\ge M\) and \(k'\).
  Furthermore, owing to the absolute continuity of the
  integral we can choose \(m\geq M\) such that \(I_{3,m} < \epsilon\).
  Then, for \(m\) fixed we can select \(K\) large enough so that for all \(k'\geq K\) we have the inequality
  \(I_{1,m,k'}<\epsilon\), thereby showing that \(\qnl_{k'} \rightharpoonup \hat{\qnl}\)
  in \(L^1(\Od\times\Od)\).
  The anti-symmetry of \(\hat{\qnl}\) is established by considering symmetric test
  functions \(\pnl\).
  Finally, \[\begin{aligned}\|\hat{\qnl}\|_{L^{1,2}(\Od\times\Od)}
  &= \limsup_{m\to\infty} \|\hat{\qnl}\|_{L^1(\Od\setminus E_m;L^2(\Od))}
  \\&\leq \limsup_{m\to\infty} \liminf_{k'\to\infty}\|\qnl_{k'}\|_{L^1(\Od\setminus E_m;L^2(\Od))}
  \\&\leq \liminf_{k'\to\infty}\|\qnl_{k'}\|_{L^1(\Od;L^2(\Od))}
  \leq 1,\end{aligned}\]
  thereby concluding the proof.
\end{proof}

Next we deal with the continuity of the divergence operator with respect to the
biting convergence.

\begin{proposition}\label{prop:divcont}
  The linear operator \(\diverx: L^{1,2}_a(\Od\times\Od) \to [W^{1,\infty}_0(\O)]'\) defined by~\eqref{eq:divx}
  transforms bounded biting-convergent sequences in \(L^{1,2}_a(\Od\times\Od)\) into weak\(^*\) convergent
  sequences in \([W^{1,\infty}_0(\O)]'\).
\end{proposition}
\begin{proof}
  Assume that \(B_{L^{1,2}_a(\Od\times\Od)} \ni \qnl_k \bite \hat{\qnl}\) as \(k\to\infty\),
  and
  \(\{E_m\}\) is the vanishing non-increasing sequence of measurable
  subsets of \(\Od\) as discussed in Definition~\ref{def:bite}.
  Let \(v \in W^{1,\infty}_0(\O)\) be arbitrary.
  Then, for each \(k,m=1,2,\dots\) we have the estimate
  \[
  \begin{aligned}
    |\langle \diverx (\qnl_k-\hat{\qnl}),v \rangle|
    &=
    \bigg|
    \int_{\Od}\int_{\Od} [\qnl_k(x,x')-\hat{\qnl}(x,x')]\grad v(x,x')\dx'\dx
    \bigg|
    \\&\leq
    \bigg|
    \int_{\Od\setminus E_m}\int_{\Od} [\qnl_k(x,x')-\hat{\qnl}(x,x')]\grad v(x,x')\dx'\dx
    \bigg|
    \\&+
    \bigg|
    \int_{E_m}\int_{\Od} \qnl_k(x,x')\grad v(x,x')\dx'\dx
    \bigg|
    \\&+
    \bigg|
    \int_{E_m}\int_{\Od} \hat{\qnl}(x,x')\grad v(x,x')\dx'\dx
    \bigg|
    \\&=|I_{1,m,k}| + |I_{2,m,k}| + |I_{3,m}|.
  \end{aligned}
  \]
  Let us fix an arbitrary \(\epsilon>0\).
  Owing to the continuity of the integral we can select \(M\in\N\) such that
  \(\forall m\geq M\) we have \(|I_{3,m}|<\epsilon\).
  We then utilize the antisymmetry of both \(\qnl_k\) and \(\grad v\)
  to change the integration order in \(I_{2,m,k}\),
  and apply Cauchy--Bunyakovsky--Schwarz and then H{\"o}lder inequalities as follows:
  \[
  \begin{aligned}
    |I_{2,m,k}|
    &=
    \bigg|\int_{\Od} \int_{E_m}\qnl_k(x,x')\grad v(x,x')\dx'\dx\bigg|
    \\&\leq
    \int_{\Od} \bigg[\int_{E_m}|\qnl_k(x,x')|^2\dx'\bigg]^{1/2}
    \bigg[\int_{E_m}|\grad v(x,x')|^2\dx']\bigg]^{1/2}\dx
    \\&\leq
    \underbrace{\|\qnl_k\|_{L^{1,2}(\Od\times\Od)}}_{\leq 1}
    \|v\|_{W^{1,\infty}(\O)}\sup_{F:F\subseteq \Rn, |F|\leq |E_m|}
    \bigg[\int_{F} \Wd^2(z)|z|^2\,{\mathrm d}z\bigg]^{1/2},
  \end{aligned}
  \]
  where we have utilized~\eqref{eq:grad}.
  We now select and fix \(m\geq M\) such that \(|I_{2,m,k}|<\epsilon\),
  \(\forall k \in \N\).
  With such an \(m\) fixed we can find \(K\in\N\) such that \(\forall k\geq K\)
  we have \(|I_{1,m,k}|<\epsilon\), owing to the biting convergence.
  Thus \(|\langle \diverx (\qnl_k-\hat{\qnl}),v \rangle|<3\epsilon\), \(\forall k\geq K\),
  and therefore \(\diverx \qnl_k \wsto \diverx \hat{\qnl}\) in \([W^{1,\infty}_0(\O)]'\).
\end{proof}

Together, Propositions~\ref{prop:compact} and~\ref{prop:divcont} imply the following theorem.

\begin{theorem}\label{thm:existence}
  The problem~\eqref{eq:prob_nl} admits an optimal solution.
\end{theorem}
\begin{proof}
  Note that for each \(f\in L^2(\O)\), Proposition~\ref{prop:diver0} implies
  that \(\Qdx^a(f)\neq\emptyset\).
  Then, Propositions~\ref{prop:compact} and~\ref{prop:divcont} show that the lower
  level sets of the objective function \(\|\cdot\|_{L^{1,2}(\Od\times\Od)}\) over \(\Qdx^a(f)\)
  are sequentially compact with respect to the biting convergence.
\end{proof}

\subsection{Fenchel dual of~\eqref{eq:prob_nl} and existence of optimal temperatures}
\label{sec:fenchel_duals}

Our analysis of the approximations to the problem~\eqref{eq:prob_nl} will proceed along the lines outlined in Section~\ref{sec:analysis_loc}.
To this end we would like to compute the Fenchel dual of~\eqref{eq:prob_nl}.
Whereas we offer no explicit formula for this problem, we will compute upper and lower estimates, which will prove to be useful in what follows.

We will benefit from being able to relate the adjoint operators for \(\diver\) and \(\diverx^*\) with the nonlocal gradient operator \(\grad\) as defined by~\eqref{eq:defgrad}.
In particular, we have the following analogue of~\cite[Proposition~9.18]{brezis2010functional}.
\begin{proposition}\label{prop:adj}
  Let us fix an arbitrary \(u \in L^2(\O)\).
  \begin{enumerate}
      \item The function \(u\) belongs to \(\Udz\) if and only if there is a constant \(c \geq 0\), such that the inequality \(|(u, \diver\qnl)_{L^2(\O)}| \leq c \|\qnl\|_{L^2(\Od\times\Od)}\) holds for all \(\qnl \in \Qd^a\).
      If this is the case, we have the equality \((u, \diver\qnl)_{L^2(\O)} = -(\grad u, \qnl)_{L^2(\Od\times\Od)}\), \(\forall \qnl \in \Qd\).
       \item The function \(\grad u\) belongs to \(L^{\infty,2}(\Od\times\Od)\) if and only if there is a constant \(c \geq 0\), such that the inequality \((u, \diverx\qnl)_{L^2(\O)} \leq c \|\qnl\|_{L^{1,2}(\Od\times\Od)}\) holds for all \(\qnl \in \Qdx\).
      If this is the case, we have the equality
      \[(u, \diverx\qnl)_{L^2(\O)} = -\int_{\Od}\int_{\Od}\grad u(x,x')\qnl(x,x')\dx\dx', \qquad \forall \qnl \in \Qdx.\]
  \end{enumerate}
\end{proposition}
\begin{proof}
    The ``only if'' part in both cases follows directly from the definition of \(\diver\) as the adjoint operator to \(\grad\); specifically see~\eqref{eq:diver} and~\eqref{eq:divx}. Therefore we only prove the ``if'' part of the claims, which we begin with the following observation.
      For each \(\qnl \in C^{0,1}(\Od\times\Od)\subset \Qd\subset\Qdx\) we have an explicit expression for \(\diver \qnl\), see~\cite[Proposition~4.1]{eb2021}, and therefore
      \[\begin{aligned}&(u, \diverx\qnl)_{L^2(\O)}=(u, \diver\qnl)_{L^2(\O)} \\&= \int_{\O}u(x)\int_{\Od}[\qnl(x',x)-\qnl(x,x')]\Wd(x-x')\dx'\dx,
      \qquad \forall \qnl \in C^{0,1}(\Od\times\Od).\end{aligned}\]
      We now extend \(u\) by zero outside of \(\O\) and integrate by parts in the equality above:
    \begin{equation*}
      \begin{aligned}
      &\int_{\O} u(x)\int_{\Od}[\qnl(x',x)-\qnl(x,x')]\Wd(x-x')\dx'\dx
      \\&=
      \lim_{\epsilon\searrow 0} \iint_{O_\epsilon} \bigg[\frac{\qnl(x',x)}{|x-x'|}u(x)|x-x'|\Wd(x-x')
      -
      \frac{\qnl(x,x')}{|x-x'|}u(x)|x-x'|\Wd(x-x')\bigg]\dx'\dx
      \\&=
      \lim_{\epsilon\searrow 0} \iint_{O_\epsilon} \bigg[\frac{\qnl(x,x')}{|x-x'|}u(x')|x-x'|\Wd(x-x')
      -
      \frac{\qnl(x,x')}{|x-x'|}u(x)|x-x'|\Wd(x-x')\bigg]\dx'\dx
      \\&=
      -\lim_{\epsilon\searrow 0} \iint_{O_\epsilon} \qnl(x,x')[u(x)-u(x')]\Wd(x-x')\dx'\dx
      \\&=-\lim_{\epsilon\searrow 0} \iint_{O_\epsilon} \qnl(x,x')\grad u(x,x')\dx'\dx.
      \end{aligned}
    \end{equation*}
    where \(O_{\epsilon} = \{\,(x,x')\in \Od\times\Od: |x-x'|>\epsilon \,\}\). Notice that we have applied Fubini's theorem. We now define a linear functional \(F_u: C^{0,1}(\Od\times\Od)\to \R\) by
    \begin{equation}\label{gradidentific}\begin{aligned}
    F_u(\qnl) &= -\lim_{\epsilon\searrow 0} \iint_{O_\epsilon} \qnl(x,x')\grad u(x,x')\dx'\dx
    =
    (u, \diver\qnl)_{L^2(\O)}=(u, \diverx\qnl)_{L^2(\O)}.\end{aligned}
    \end{equation}
  \begin{enumerate}
      \item
      Owing to our assumptions, \eqref{gradidentific} defines a bounded linear functional on a dense subspace \(C^{0,1}(\Od\times\Od)\cap L^2_a(\Od\times\Od)\) of the Hilbert space \(L^2_a(\Od\times\Od)\).
      Therefore, we use Riesz' representation theorem to conclude that there exists a function \(\pnl \in L^2_a(\Od\times\Od)\)
      such that \((\pnl,\qnl)_{L^2(\Od\times\Od)} = F_u(q)\), for all \(\qnl\in\Qd^a\).
      Also note that owing to the anti-symmetry \(\grad u\) the equality still holds for all \(\qnl\in C^{1,0}(\Od\times\Od)\) regardless of the symmetry of \(\qnl\), since anti-symmetric and symmetric functions are \(L^2(\Od\times\Od)\) are orthogonal, see~\eqref{eq:orth}.
    Appealing to~\eqref{gradidentific} once again, we can see that \(\pnl = \grad u\), almost everywhere in \(O_\epsilon\), \(\epsilon>0\), and consequently also in \(\Od\times\Od\).
    Therefore \(u\in \Udz\) as claimed.  The equality \((u, \diver\qnl)_{L^2(\O)} = -(\grad u, \qnl)_{L^2(\Od\times\Od)}\), \(\forall \qnl \in \Qd\) follows from the definition of the operator \(\diver\).
    \item
    Owing to our assumptions, \eqref{gradidentific} defines a bounded linear functional on a dense subspace \(C^{0,1}(\Od\times\Od)\) of
      \(L^{1,2}(\Od\times\Od)\supset \Qdx\).
      Since \(L^{\infty,2}(\Od\times\Od)\) is isometrically isomorphic to the dual of \(L^{1,2}(\Od\times\Od)\) with double integration as the primal-dual pairing, there is a function \(\pnl \in L^{\infty,2}(\Od\times\Od)\) such that
    \(( u, \diverx\qnl )_{L^2(\O)}= \int_{\Od} \int_{\Od} \pnl(x,x')\qnl(x,x')\dx\dx'\), for all \(\qnl\in\Qdx\).
    Appealing to~\eqref{gradidentific} once again, we can see that \(\pnl = \grad u\), almost everywhere in \(O_\epsilon\), \(\epsilon>0\), and consequently also in \(\Od\times\Od\).
  \end{enumerate}
\end{proof}

We now proceed with the discussion of the dual for~\eqref{eq:prob_nl}.   To this end, we observe that the conditions for strong Fenchel duality are fulfilled by~\eqref{eq:prob_nl}: indeed
\(\diverx \in \mathcal{B}(\Qdx^a; L^2(\O))\) is a surjective bounded linear operator even when restricted to the smaller subspace of
anti-symmetric fluxes \(\Qdx^a\) (or even \(\Qd^a\)), as established in Proposition~\ref{prop:diver0}, and both \(\|\cdot\|_{L^{1,2}(\Od\times\Od)}\) and \(\ind_{\{f\}}\) are convex and lower-semicontinuous on \(\Qdx^a\) and \(L^2(\O)\), respectively.
Whereas the convex conjugate of the second term, \(\ind^*_{\{f\}}(u)=-\ell(u)\) has been previously utilized in Section~\ref{sec:analysis_loc}, the computation of the convex conjugate of the first term is a little subtle owing to the requirement that the supremum is taken over \(\Qdx^a \subset \Qdx\), that is, over anti-symmetric fluxes only:
\begin{equation}\label{eq:conj2}\begin{aligned}
\ind_{B^*_{L^{1,2}_a(\Od\times\Od)}}(\pnl)
&= \sup_{\qnl \in \Qdx^a}[ \langle \pnl, \qnl \rangle - \|\qnl\|_{L^{1,2}(\Od\times\Od)}]
\end{aligned}\end{equation}
where we have introduced the notation \(\ind_{B^*_{L^{1,2}_a(\Od\times\Od)}}\) for this conjugate.
Indeed, because \(\|\cdot\|_{L^{1,2}(\Od\times\Od)}\) is zero at \(0\in L^{1,2}_a(\Od\times\Od)\) and is absolutely homogeneous, the notation is warranted since \(\ind_{B^*_{L^{1,2}_a(\Od\times\Od)}}\) assumes only values in the set \(\{0,+\infty\}\).
With this notation in place, let us consider the following problems:
\begin{align}\label{eq:limiting_nl_dual_asym}
  -d^*_{\delta,a}&\overset{\cdot}{=}
  \inf_{u\in L^2(\O)}[ \ind_{B^*_{L^{1,2}_a(\Od\times\Od)}}(\diverx^* u) - \ell(u)],
\intertext{and}
\label{eq:limiting_nl_dual_2}
  -d^*_{\delta,2}&\overset{\cdot}{=}
  \inf_{u\in L^2(\O)}[\ind_{|\Od|^{1/2}B_{L^{2}(\Od\times\Od)}}(\grad u) - \ell(u)],
\end{align}
where~\eqref{eq:limiting_nl_dual_asym} is the Fenchel dual of~\eqref{eq:prob_nl}
and \(\diverx^* \in \mathcal{B}(L^2(\O);(\Qdx^a)')\) is understood as the negative adjoint operator for \(\diverx\) restricted to anti-symmetric fluxes.

\begin{proposition}\label{prop:infima_ineq}
  Infima~\eqref{eq:limiting_nl_dual_asym} and~\eqref{eq:limiting_nl_dual_2} are related as follows:
  \begin{equation}\label{eq:dual_finite2}
   0\leq p^*_{\delta,a} = d^*_{\delta,a} \leq d^*_{\delta,2} < \infty.
  \end{equation}
  Both infima are attained.
\end{proposition}
\begin{proof}
The inequality \(0\leq p^*_{\delta,a}\) follows trivially from the non-negativity of the objective function of~\eqref{eq:prob_nl},
the equality \(p^*_{\delta,a} = d^*_{\delta,a}\) is the expression of the strong Fenchel duality, and
the inequality \(p^*_{\delta,a}<+\infty\) is a consequence of Theorem~\ref{thm:existence}.
Therefore, \(d^*_{\delta,a}\) is finite and the infimum is attained, see~\cite[Theorem 4.4.3]{borwein2004techniques}.

The inequality \(d^*_{\delta,2} < \infty\) follows from the fact that the infimum in~\eqref{eq:limiting_nl_dual_2} is attained.
Indeed, the objective function of~\eqref{eq:limiting_nl_dual_2} is inf-compact in the weak topology of \(L^2(\O)\), as follows from the Poincar{\'e}-type inequality, see Proposition~\ref{prop:grad0}.

It remains to establish the inequality \(d^*_{\delta,a} \leq d^*_{\delta,2}\), for which in turn it is sufficient to prove the estimate
\(\ind_{|\Od|^{1/2}B_{L^{2}(\Od\times\Od)}}(\grad u) \leq \ind_{B^*_{L^{1,2}_a(\Od\times\Od)}}(\diverx^* u)\), for an arbitrary \(u\in L^2(\O)\).
We proceed as follows:
\begin{equation*}
  \begin{aligned}
    \ind_{B^*_{L^{1,2}_a(\Od\times\Od)}}(\diverx^* u)
    & = \sup_{\qnl \in \Qdx^a}[-(\diverx\qnl,u)_{L^2(\O)} - \|\qnl\|_{L^{1,2}(\Od\times\Od)}]
    \\&\geq \sup_{\qnl \in \Qd^a}[-(\diver\qnl,u)_{L^2(\O)} - |\Od|^{1/2}\|\qnl\|_{L^{2}(\Od\times\Od)}],
  \end{aligned}
\end{equation*}
where the inequality is owing to the inclusion \(\Qd^a\subset\Qdx^a\), the fact that \(\diverx|_{\Qd} = \diver\),
and H{\"o}lder's inequality applied to \(L^{1,2}(\Od\times\Od)\)-norm.
Note that if there is a sequence \(\qnl_1,\qnl_2,\ldots \subset \Qd^a\) such that \(-(\diver\qnl_k,u)_{L^2(\O)}>k\|\qnl\|_{L^{2}(\Od\times\Od)}\), then the supremum equals \(+\infty\).
Otherwise \(|(\diver\qnl,u)_{L^2(\O)}| \leq c\|\qnl\|_{L^{2}(\Od\times\Od)}\) for some \(c>0\) and all \(\qnl \in \Qd^a\),
which in view of Proposition~\ref{prop:adj}, point 1 is equivalent with the assertion \(\grad u \in L^2(\Od\times\Od)\).
In either case, the supremum equals \(\ind_{|\Od|^{1/2}B_{L^{2}(\Od\times\Od)}}(\grad u)\), and the proof is concluded.
\end{proof}

Note that the upper bound \(p^*_{\delta,a}\leq d^*_{\delta,2}\) in~\eqref{eq:dual_finite2} combined with Poincar{\'e}-type inequality for \(\grad\), see Proposition~\ref{prop:grad0}, implies that the \(L^{1,2}(\Od\times\Od)\)-norms of optimal solutions to \eqref{eq:prob_nl} remain uniformly bounded with respect to \(\delta\searrow 0\),
for each \(f\in L^2(\O)\).

\subsection{Tikhonov regularization of~\eqref{eq:prob_nl}}\label{sec:tikhonov}

We now turn to the analysis of the Tikhonov \(L^2(\Od\times\Od)\)-regularization
of~\eqref{eq:prob_nl}.
These developments will be useful in establishing a rigorous link between~\eqref{eq:prob_nl} and~\eqref{eq:compliance_nl_mixed}.
We follow the same approach as in Section~\ref{sec:analysis_loc}.

For \(\beta > 0\) let us define a convex functional \(\Jdb:\Qd\to\R\)
by \(\Jdb(\qnl) = \|\qnl\|_{L^{1,2}(\Od\times\Od)} + \frac{\beta}{2}\|\qnl\|^2_{L^{2}(\Od\times\Od)}\).
We then consider the pair of primal
\begin{align}\label{eq:tikhonov}
  p^*_{\delta,\beta} &= \inf_{\qnl\in\Qd^a} [\Jdb(\qnl)  + \ind_{\{f\}}(\diver\qnl)]\\
 \intertext{and dual}
\label{eq:tikhonov_dual}
  -d^*_{\delta,\beta} &= \inf_{u\in L^2(\O)}[\Jdb^*(\diver^* u) - \ell(u)]
\end{align}
problems, where \(\Jdb^*: (\Qd^a)'\to \R\cup\{+\infty\}\) is the convex conjugate of \(\Jdb\):
\begin{equation}\label{eq:jdbc}
  \Jdb^*(\pnl) = \sup_{\qnl \in \Qd^a}[ \langle \pnl, \qnl \rangle - \Jdb(\qnl)],
\end{equation}
and as before \(\diver^*\in \mathcal{B}(L^2(\O), (\Qd^a)')\) is the negative adjoint of \(\diver\).

\begin{proposition}\label{prop:tikh_exist}
 The following statements hold.
  \begin{enumerate}
    \item For each \(\beta>0\) the infimum in~\eqref{eq:tikhonov} is uniquely attained.
    \item For each \(\beta>0\), the strong duality holds: \(p^*_{\delta,\beta} = d^*_{\delta,\beta}\),
    and the infimum in~\eqref{eq:tikhonov_dual} is attained.
    \item For each \(\overline{\beta}>0\), the family \(\Jdb^*(\diver^*\cdot)\), \(\beta \in [0,\overline{\beta})\) is equi-coercive on \(L^2(\O)\).
    \item
      The functionals \(\Jdb^*(\diver^* \cdot): L^2(\Od)\to \R\cup \{+\infty\}\),
      \(\Gamma\)-converge to
      \(\ind_{B^*_{L^{1,2}_a(\Od\times\Od)}}(\diverx^* \cdot): L^2(\Od)\to \R\cup \{+\infty\}\), as \(\beta\searrow 0\).
   \item
    We have the equality \(\lim_{\beta\searrow 0} d^*_{\delta,\beta} = d^*_{\delta,a}\), and
    any \(L^2(\Od)\)-limit point of solutions to~\eqref{eq:tikhonov_dual} is a solution
    to~\eqref{eq:limiting_nl_dual_asym}.
 \end{enumerate}
\end{proposition}
\begin{proof}
  \begin{enumerate}
    \item
    The problem~\eqref{eq:tikhonov} amounts to minimizing a convex, lower semicontinuous, and coercive functional \(\Jdb(\cdot)  + \ind_{\{f\}}(\diver\cdot)\) over the Hilbert space \(\Qd^a\).
    The solution exists owing to the generalized Weierstrass' theorem, see for example~\cite[Corollary~3.23]{brezis2010functional}.
    %
   \item
    All conditions for strong duality are fulfilled, with the least trivial one
    being the surjectivity of \(\diver\) on the space of anti-symmetric functions, which is established in Proposition~\ref{prop:diver0}.
  \item
    Let us fix an arbitrary \(\overline{\beta}>0\).
    We begin by stating the upper estimate \(\Jdb(\qnl) \leq |\Od|^{1/2}\|\qnl\|_{L^2(\Od\times\Od)} + \frac{\overline{\beta}}{2}\|\qnl\|^2_{L^2(\Od\times\Od)}\),
    which holds for each \(\qnl \in \Qd\) and \(\beta \in [0,\overline{\beta})\).
    Arguing as in Proposition~\ref{prop:infima_ineq}, we have the inequality
    \begin{equation*}
      \Jdb^*(\diver^* u) \geq
      \sup_{\qnl\in \Qd^a}\left[-( u,\diver\qnl)_{L^2(\O)}-|\Od|^{1/2}\|\qnl\|_{L^2(\Od\times\Od)} - \frac{\overline{\beta}}{2}\|\qnl\|^2_{L^2(\Od\times\Od)}\right].
    \end{equation*}
    If there is a sequence \(\qnl_1,\qnl_2,\ldots\subset\Qd^a\) such that \(-( u,\diver\qnl_k)_{L^2(\O)} > k \|\qnl_k\|_{L^2(\Od\times\Od)}\),
    the supremum equals \(+\infty\), as is easily seen by evaluating the terms along the renormalized sequence \(\|\qnl_k\|_{L^2(\Od\times\Od)}^{-1}\qnl_k\).
    Otherwise there is a constant \(c>0\) such that \(|( u,\diver\qnl)_{L^2(\O)}| \leq c\|\qnl\|_{L^2(\Od\times\Od)}\), for each \(\qnl\in\Qd^a\),
    which owing to Proposition~\ref{prop:adj}, point 1, is equivalent with the assertion \(u\in\Udz\).
    In this case \(-( u,\diver\qnl)_{L^2(\O)}=(\grad u,\qnl)_{L^2(\Od\times\Od)}\),
    and therefore we continue the estimate
    \begin{equation*}\begin{aligned}
      \Jdb^*(\diver^* u) &\geq
      \sup_{\qnl\in \Qd^a}\left[(\grad u,\qnl)_{L^2(\Od\times\Od)}-|\Od|^{1/2}\|\qnl\|_{L^2(\Od\times\Od)} - \frac{\overline{\beta}}{2}\|\qnl\|^2_{L^2(\Od\times\Od)}\right]
      \\&=
      \sup_{\qnl\in L^2_a(\Od\times\Od)}\left[(\grad u,\qnl)_{L^2(\Od\times\Od)}-|\Od|^{1/2}\|\qnl\|_{L^2(\Od\times\Od)} - \frac{\overline{\beta}}{2}\|\qnl\|^2_{L^2(\Od\times\Od)}\right]
      \\&=
      -\inf_{\qnl\in L^2_a(\Od\times\Od)}\left[|\Od|^{1/2}\|\qnl\|_{L^2(\Od\times\Od)} + \frac{\overline{\beta}}{2}\|\qnl-\overline{\beta}^{-1}\grad u\|^2_{L^2(\Od\times\Od)}\right] + \frac{1}{2\overline{\beta}}\|\grad u\|^2_{L^2(\Od\times\Od)},
    \end{aligned}\end{equation*}
    where the first equality is owing to the density of \(\Qd^a\) in \(L^2_a(\Od\times\Od)\).
    The infimum is attained at \(\hat{\qnl} = \overline{\beta}^{-1}\max\{1-|\Od|^{1/2}\|\grad u\|^{-1}_{L^2(\Od\times\Od)},0\}\grad u\).
    The two cases can be combined in one lower estimate:
    \begin{equation*}
      \Jdb^*(\diver^* u) \geq
      \frac{1}{2\overline{\beta}}
      [\max\{\|\grad u\|_{L^2(\Od\times\Od)}-|\Od|^{1/2},0\}]^2.
    \end{equation*}
    Together with Poincar{\'e}-type inequality for \(\grad\), see Proposition~\ref{prop:grad0}, this implies the claimed equi-coercivity.
 \item
The family of functionals \(\Jdb(\cdot)\) is monotonically non-decreasing with respect to \(\beta\), and therefore owing to the order-reversing property of convex conjugates the family of functionals \(\Jdb^*(\cdot)\) is monotonically non-increasing with respect to \(\beta\). Furthermore, convex conjugates are lower semicontinuous, and \(\diver^*\) is an adjoint of a linear bounded operator, consequently the composite functionals \(\Jdb^*(\diver^*\cdot)\) are lower semicontinuous. Therefore, in view of~\cite[Remark 2.12]{braides2006handbook} the \(\Gamma\)-limit of \(\Jdb^*(\diver^*\cdot)\) is given by the pointwise limit of these functionals.

The pointwise limit is computed as follows.
  Using the definition of \(\Jdb\), in particular the facts that \(\Jdb(0)=0\) and \(\Jdb(\qnl)\geq \|\qnl\|_{L^{1,2}(\Od\times\Od)}\),
  \(\forall \beta >0\) and \(\qnl\in L^{2}(\Od\times\Od)\),  we have the following estimate:
  \[\begin{aligned}
  0 \leq \Jdb^*(\diver^* u) &= \sup_{\qnl\in\Qd^a}[ -(u,\diver\qnl)_{L^2(\O)} - \Jdb(\qnl)]
  \\&\leq \sup_{\qnl\in\Qdx^a}[ -(u,\diverx\qnl)_{L^2(\O)} - \|\qnl\|_{L^{1,2}(\Od\times\Od)}]
  \\&=\ind_{B_{L^{1,2}_a(\Od\times\Od)}^*}(\diverx^* u).
  \end{aligned}\]
  We will now show that for each \(u\in L^2(\O)\) with \(\diverx^* u \not\in B_{L^{1,2}_a(\Od\times\Od)}^*\) we have
  the equality \(\lim_{\beta\searrow 0} \Jdb^*(\diver^* u) =+\infty\).
  Let us fix such a \(u\in L^2(\O)\). Per  (the implicit) definition of \(B_{L^{1,2}_a(\Od\times\Od)}^*\), for each \(N\in\N\) there is \(\qnl_N \in \Qdx^a\) such that
  \(\langle \diverx^* u,\qnl_N\rangle - \|\qnl_N\|_{L^{1,2}(\Od\times\Od)}=-(u,\diverx\qnl_N)_{L^2(\O)} - \|\qnl_N\|_{L^{1,2}(\Od\times\Od)}\geq N\).

  Let us take an arbitrary \(\phi \in C_c^\infty(B_{\Rn})\), \(\phi\geq 0\), \(\int_{\Rn}\phi(x)\dx=1\) and construct a family of mollifiers \(\phi_\epsilon=\epsilon^{-n}\phi(\epsilon^{-1}x)\).
  We extend \(\qnl_N\) by \(0\) outside of \(\Od\times\Od\), and put
  \[
  \qnl_{N,\epsilon}(x,x') = \int_{\Rn}\int_{\Rn}\phi_\epsilon(x-y)\phi_\epsilon(x'-y')\qnl_N(y,y')\,\mathrm{d}y\,\mathrm{d}y'.
  \]
  Then \(\qnl_{N,\epsilon}\in C^\infty_c(\Rn\times\Rn)\), is anti-symmetric, and satisfies the equality
  \(\lim_{\epsilon\searrow 0} \|\qnl_{N,\epsilon}-\qnl_N\|_{L^{1,2}(\Od\times\Od)}=0.\)
  Additionally, for each \(v \in C^\infty_c(\O)\) we have \(\grad v \in L^{\infty,2}(\Od\times\Od)\), and consequently
  \[
  \begin{aligned}
  \lim_{\epsilon\searrow 0} (v,\diverx\qnl_{N,\epsilon})_{L^2(\O)}&=
  \lim_{\epsilon\searrow 0} (v,\diver\qnl_{N,\epsilon})_{L^2(\O)} =
  -\lim_{\epsilon\searrow 0} \int_{\Od}\int_{\Od}\qnl_{N,\epsilon}(x,x')\grad v(x,x')\dx\dx'
  \\&=-\int_{\Od}\int_{\Od}\qnl_{N}(x,x')\grad v(x,x')\dx\dx' = (v,\diverx\qnl_N).
  \end{aligned}
  \]
  Since both the leftmost and the rightmost hand sides are continuous in \(L^2(\O)\) with respect to \(v\), we conclude that their equality holds for each \(v\in L^2(\O)\).

  We now choose \(\epsilon_N\) so that the inequalities \(\|\qnl_{N,\epsilon}-\qnl_N\|_{L^{1,2}(\Od\times\Od)}<\frac{1}{3}\) and \(|(u,\diverx\qnl_{N,\epsilon_N}-\diverx\qnl_N)_{L^2(\O)}| <\frac{1}{3}\) hold.
  We then choose \(\beta_{N}\) so small that \(\frac{\beta_{N}}{2}\|\qnl_{N,\epsilon_N}\|^2_{L^{2}(\Od\times\Od)}<\frac{1}{3}\).
  With these choices, for all \(\beta < \beta_{N}\) we have a lower estimate:
  \begin{equation*}\begin{aligned}
  \Jdb^*(\diver^* u)&\geq
  -(u, \diver\qnl_{N,\epsilon_N})
  - \|\qnl_{N,\epsilon_N}\|_{L^{1,2}(\Od\times\Od)}
  -\frac{\beta}{2}\|\qnl_{N,\epsilon_N}\|^2_{L^{2}(\Od\times\Od)}
  > N-1,
  \end{aligned}\end{equation*}
  thereby concluding the proof.
  \item
  Note that \(\Jdb^*(\diver^*\cdot)-\ell\) is still equi-coercive on \(L^2(\O)\) and \(\Gamma\)-converges to \(\ind_{B^*_{L^{1,2}_a(\Od\times\Od)}}(\diverx^*\cdot) - \ell\), since \(\Gamma\)-convergence is stable with respect to continuous perturbations, see for example~\cite[Remark 2.2]{braides2006handbook}.
   Therefore, in view of point 4, it remains to appeal to the fundamental theorem of \(\Gamma\)-convergence, see for example~\cite[Theorem 2.10]{braides2006handbook}.
  \end{enumerate}
\end{proof}

The convergence of the two-point optimal fluxes follows from these auxiliary results.
\begin{theorem}\label{thm:tikhonov}
  Let \(\hat{\qnl}_{\beta} \in \Qd^a(f)\) be the family of solutions to~\eqref{eq:tikhonov},
  \(\beta \in (0,\overline{\beta}]\).
  The following statements hold:
  \begin{enumerate}
    \item
    \(p^*_{\delta,a}=\lim_{\beta\searrow 0} p^*_{\delta,\beta}\).
    In fact, \(p^*_{\delta,a}=\lim_{\beta\searrow 0} \|\hat{\qnl}_{\beta}\|_{L^{1,2}(\Od\times\Od)}\),
    and \(\lim_{\beta\searrow 0} (\beta/2)\|\hat{\qnl}_{\beta}\|^2_{L^{2}(\Od\times\Od)}=0\).
    \item The family \(\hat{\qnl}_{\beta}\)
    is relatively sequentially compact with respect to the biting convergence in \(L^{1}(\Od;L^2(\Od))\).
    Each limit point of this sequence is a solution to~\eqref{eq:prob_nl}.
  \end{enumerate}
\end{theorem}
\begin{proof}
  \begin{enumerate}
    \item
    As each \(\hat{\qnl}_{\beta}\) is feasible in~\eqref{eq:prob_nl}, we have the estimate
    \begin{equation}\label{eq:thm60}\begin{aligned}
      p^*_{\delta,a} &\leq \liminf_{\beta\searrow 0} \|\hat{\qnl}_{\beta}\|_{L^{1,2}(\Od\times\Od)}
      \leq \limsup_{\beta\searrow 0} \|\hat{\qnl}_{\beta}\|_{L^{1,2}(\Od\times\Od)}
      \\&\leq
      \lim_{\beta\searrow 0} \left[\|\hat{\qnl}_{\beta}\|_{L^{1,2}(\Od\times\Od)}+\frac{\beta}{2}\|\hat{\qnl}_{\beta}\|^2_{L^{2}(\Od\times\Od)}\right]
      = \lim_{\beta\searrow 0}p^*_{\delta,\beta}
      = p^*_{\delta,a},
    \end{aligned}\end{equation}
    where the last equality follows from Proposition~\ref{prop:tikh_exist}, points 2 and 5, and Proposition~\ref{prop:infima_ineq}.
    Therefore, the equalities must hold throughout, and the claim follows.
    \item
    Note that~\eqref{eq:thm60} implies that the norms \(\|\hat{\qnl}_{\beta}\|_{L^{1,2}(\Od\times\Od)}\) remain bounded as \(\beta\searrow 0\).
    Consequently, Proposition~\ref{prop:compact} provides us with the claimed sequential compactness of this family of minimizers.
    Let \(\hat{\qnl} \in L^{1,2}_a(\Od\times\Od)\) be an arbitrary limit point of \(\hat{\qnl}_{\beta}\)
    along some sequence \(\beta_{k}\searrow 0\).
    We have the estimates
    \[
    p^*_{\delta,a} \leq \|\hat{\qnl}\|_{L^{1,2}(\Od\times\Od)}
    \leq \liminf_{\beta_k\searrow 0} \|\hat{\qnl}_{\beta_k}\|_{L^{1,2}(\Od\times\Od)}
    =p^*_{\delta,a},
    \]
    where the first inequality is implied by Proposition~\ref{prop:divcont} asserting that \(\hat{\qnl} \in \Qdx^a(f)\), the second by Proposition~\ref{prop:compact}, and the equality is provided by point 1.
    This is therefore sufficient to conclude that \(\hat{\qnl}\) is an optimal solution of~\eqref{eq:prob_nl}.
  \end{enumerate}
\end{proof}

\subsection{Relationship between~\eqref{eq:prob_nl} and~\eqref{eq:compliance_nl_mixed}}\label{sec:gammaconv1}

In this section we consider the relationship between~\eqref{eq:prob_nl} and~\eqref{eq:compliance_nl_mixed}
(or equivalently, \eqref{eq:igamma_nl}).
We remind the reader that in order to do this we assume that  \(\underline{\kappa}=0\).%
We have the following result.

\begin{theorem}\label{thm:vanish_limit}
  The family of optimal solutions \(\qnl_\gamma \in \Qdx^a(f)\),
  \(\gamma\searrow 0\), to~\eqref{eq:igamma_nl}
  is relatively sequentially compact with respect to the biting convergence in  \(L^{1}(\Od;L^2(\Od))\).
  Each limit point of this sequence is an optimal solution to~\eqref{eq:prob_nl}.
\end{theorem}
\begin{proof}
  We will prove the following equality:
  \begin{equation}\label{thm70}
    \lim_{\gamma\searrow 0} \hat{p}^*_{\delta,\gamma}
    = \frac{1}{2|\Od|}[p^*_{\delta,a}]^2,
  \end{equation}
  where expression in the right hand side stems from the relation between \(i^0_\delta(\cdot)\) and \(\|\cdot\|_{L^{1,2}(\Od\times\Od)}\) established in Proposition~\ref{prop:id0nl} and the discussion immediately thereafter.

  For each \(\gamma > 0\) we have the inclusion \(\kadmd^\gamma\subseteq \kadmd^0\), which in turn implies the inequality \(i^0_\delta(\qnl)\leq  i^\gamma_\delta(\qnl)\),
  \(\forall \qnl \in \Qd\).
  Therefore, in view of Proposition~\ref{prop:id0nl}, and noting that each \(\qnl_\gamma\) is feasible for~\eqref{eq:prob_nl}, we have the following inequality
  \begin{equation}\label{thm71}
  \frac{1}{2|\Od|}[p^*_{\delta,a}]^2 \leq
  \frac{1}{2|\Od|}\|\qnl_\gamma\|_{L^{1,2}(\Od\times\Od)}^2
  = i^0_\delta(\qnl_\gamma) \leq i^\gamma_\delta(\qnl_\gamma)
  =\hat{p}^*_{\delta,\gamma}.
\end{equation}
We will now estimate \(\hat{p}^*_{\delta,\gamma}\) from above using Tikhonov regularization problems discussed in the previous section.
Let us put \([\kadmd^\gamma]^{-1} = \{\, \kappa^{-1} \mid \kappa \in \kadmd^\gamma\,\}\),
\(\gamma \geq 0\).
Then \(\gamma\overline{\kappa}^{-1} + [\kadmd^0]^{-1} \subseteq [\kadmd^\gamma]^{-1}\), \(\forall \gamma\geq 0\).
Therefore,
\begin{equation}\label{thm72}
\begin{aligned}
  \hat{p}^*_{\delta,\gamma}
  &= \inf_{\qnl\in\Qdx^a(f)}
  \inf_{\kappa \in \kadmd^\gamma} \frac{1}{2}\int_{\Od} \kappa^{-1}(x)\int_{\Od} |\qnl(x,x')|^2\dx'\dx
  \\&\leq \inf_{\qnl\in\Qdx^a(f)}
  \inf_{\kappa \in \kadmd^0} \frac{1}{2}\int_{\Od} [\gamma\overline{\kappa}^{-1} + \kappa^{-1}(x)]\int_{\Od} |\qnl(x,x')|^2\dx'\dx
  \\&= \inf_{\qnl\in\Qdx^a(f)}\bigg[\id^0(\qnl) + \frac{\gamma\overline{\kappa}^{-1}}{2}\|\qnl\|_{L^2(\Od\times\Od)}^2\bigg]
  \\&\leq
  \id^0(\hat{\qnl}_{\delta,\gamma\overline{\kappa}^{-1}}) + \frac{\gamma\overline{\kappa}^{-1}}{2}\|\hat{\qnl}_{\delta,\gamma\overline{\kappa}^{-1}}\|_{L^2(\Od\times\Od)}^2
  \\&=
  \frac{1}{2|\Od|}\|\hat{\qnl}_{\delta,\gamma\overline{\kappa}^{-1}}\|_{L^{1,2}(\Od\times\Od)}^2 + \frac{\gamma\overline{\kappa}^{-1}}{2}\|\hat{\qnl}_{\delta,\gamma\overline{\kappa}^{-1}}\|_{L^2(\Od\times\Od)}^2,
\end{aligned}
\end{equation}
where \(\hat{\qnl}_{\delta,\gamma\overline{\kappa}^{-1}} \in \Qds(f)\) is the unique solution to~\eqref{eq:tikhonov}
for \(\beta=\gamma\overline{\kappa}^{-1}>0\).
It remains to let \(\gamma\searrow 0\) and utilize Theorem~\ref{thm:tikhonov} point 1, to arrive at~\eqref{thm70} from~\eqref{thm71} and~\eqref{thm72}.

To conclude the proof it is sufficient to argue along the lines of Theorem~\ref{thm:tikhonov} point 2.
Indeed, \eqref{thm71} and~\eqref{thm72} imply that \(\lim_{\gamma\searrow 0} i^0_\delta(\qnl_\gamma)=
\lim_{\gamma\searrow 0} (2|\Od|)^{-1} \|\qnl_\gamma\|_{L^{1,2}(\Od\times\Od)}^2 = (2|\Od|)^{-1}[p^*_{\delta,a}]^2\),
and ultimately the existence,
feasibility, and optimality of limit points of \(\qnl_\gamma\) in~\eqref{eq:prob_nl}.
\end{proof}

\subsection{Flux recovery operator and partial nonlocal-to-local convergence results}\label{sec:gammaconv2}

In this section we discuss partial results connecting the nonlocal problem~\eqref{eq:prob_nl} towards its local counterpart~\eqref{eq:limiting_local_dual_rig}.
We follow the ideas outlined in \cite[Section~5]{eb2021}.
Towards this end, for each \(\qnl \in L^{1,2}(\Od\times\Od)\) let us define the flux recovery operator
\begin{equation}\label{eq:Rdef}
  R_\delta \qnl(x) = \int_{\Od} (x-x')\qnl(x,x')\Wd(x-x')\dx'.
\end{equation}

\begin{proposition}\label{prop:claim0}
  \(R_\delta \in \mathcal{B}(L^{1,2}(\Od\times\Od),L^1(\Od;\Rn))\)
  and \(\|R_\delta\|\leq 1\).
\end{proposition}
\begin{proof}
  Let us take an arbitrary \(\psi \in L^\infty(\O_\delta;\Rn)\). Then, owing to H{\"o}lder's inequality, we have the estimate:
  \[\begin{aligned}
  \int_{\O_\delta} R_\delta \qnl(x)\cdot \psi(x)\dx
  &=
  \int_{\Od}\int_{\Od} \psi(x)\cdot(x-x')\qnl(x,x')\Wd(x-x')\dx'\dx
  \\&\leq
  \|\qnl\|_{L^{1,2}(\Od\times\Od)}
  \|\psi(x)\cdot(x-x')\Wd(x-x')\|_{L^{\infty,2}(\Od\times\Od)} \\
  &\leq\|\qnl\|_{L^{1,2}(\Od\times\Od)} \| \psi\|_{L^\infty(\Od)},
  \end{aligned}\]
  where we have used the fact that for all \(x\in \Od\) we have the inequality
  \begin{equation}\label{eq:psi}
    \begin{aligned}
   &\int_{\Od}[\psi(x)\cdot(x-x')\Wd(x-x')]^2\dx'
   \le
    \int_{\delta B_{\Rn}} \bigg[\psi(x)\cdot\frac{z}{|z|}\bigg]^2 |z|^2\Wd^2(z)\,\mathrm{d}z
    \\
    &=\int_0^\delta r^{n-1} r^2 w^2_\delta(r) \left(\int_{\Sn} \left[ \psi(x)\cdot s\right]^2\,ds \right)\, \mathrm{d}r
    \\
    &= |\psi(x)|^2 \underbrace{\int_{\delta B_{\Rn}} |z|^2 w^2_\delta(z) \,\mathrm{d}z}_{=K_{2,n}^{-1}} \underbrace{|\Sn|^{-1} \int_{\Sn} \left[e\cdot s\right]^2\,\mathrm{d}s}_{=K_{2,n}}
    = |\psi(x)|^2.
  \end{aligned}
  \end{equation}
\end{proof}

\begin{remark}\label{prop:claim2}
  In fact, \(R_\delta \in \mathcal{B}(L^{1,2}_a(\Od\times\Od),L^1(\Od;\Rn))\)
  is weakly compact for each fixed \(\delta>0\).
  Indeed, let us consider an arbitrary \(\qnl \in B_{L^{1,2}_a(\Od\times\Od)}\),
  and a measurable subset \(E \subset \Od\).
  Then,
  \[
  \begin{aligned}
    \int_{E} |R_{\delta}\qnl(x)|\dx &\leq
    \int_{E}\int_{\Od} |\qnl(x,x')||x-x'|\Wd(x-x')\dx'\dx
    \\&=
    \int_{\Od}\int_{E} |\qnl(x,x')||x-x'|\Wd(x-x')\dx'\dx
    \\&\leq
    \underbrace{\|\qnl\|_{L^{1,2}(\Od\times\Od)}}_{\leq 1}
    \sup_{F:F\subseteq \Rn, |F|\leq |E|} \bigg[\int_{F} \Wd^2(z)|z|^2\,{\mathrm d}z\bigg]^{1/2}.
  \end{aligned}
  \]
  Therefore, \(\{\, R_{\delta}\qnl \,|\, \qnl \in B_{L^{1,2}_a(\Od\times\Od)}\,\}\) is equi-integrable,
  and is consequently weakly compact owing to Dunford--Pettis criterion~\cite[Theorem~4.30]{brezis2010functional}.
  However, equi-integrability is not uniform with respect to \(\delta\searrow 0\).
\end{remark}

\begin{proposition}\label{prop:claim1_new}
  Assume that the sequence \(\qnl_{\delta_k} \in \thickbar{\Q}_{\delta_k}(f)\), \(\delta_k \in (0,\bar{\delta})\),
  \(\lim_{k\to\infty}\delta_k = 0\), is bounded, and
  \(R_{\delta_k}\qnl_{\delta_k} \mathcal{H}^n \to \hat{q}\in \mathcal{M}(\cl\O;\Rn)\) in the sense of weak convergence.
  Then, for each \(\phi \in C^1_c(\O)\), we have the equality
  \[-\int_{\O} \nabla\phi(x)\cdot \mathrm{d}\hat{q}(x) = \ell(\phi).\]
\end{proposition}
\begin{proof}
  Throughout the proof we omit the index \(k\) and simply write \(\delta\searrow 0\),
  understanding that the limits are taken along a given subsequence.

  Recall that for each \(\qnl_{\delta} \in \Qdx(f)\) and \(\phi \in C^1_c(\O)\subset W^{1,\infty}_0(\O)\) we have the equality
  \begin{equation*}
    \ell(\phi) = \int_{\O} \diverx\qnl_{\delta}(x) \phi(x)\dx
    =-\int_{\Od}\int_{\Od} \qnl_{\delta}(x,x')\grad\phi(x,x')\dx\dx',
  \end{equation*}
  see~\eqref{eq:divx}.
  We can then write
  \begin{equation*}
      \begin{aligned}
        &\bigg|\int_{\O} \nabla\phi(x)\cdot R_{\delta}\qnl_{\delta}(x)\dx + \ell(\phi)\bigg|
        \\&= \bigg| \int_{\Od} \nabla\phi(x)\cdot R_{\delta}\qnl_{\delta}(x)\dx
        - \int_{\Od}\int_{\Od} \grad\phi(x,x') \qnl_{\delta}(x,x')\dx\dx'\bigg|\\
        &=
        \bigg|
        \int_{\Od}\int_{\Od} \{[\phi(x')-\phi(x)]-\nabla\phi(x)\cdot(x'-x)\}\qnl_{\delta}(x,x')\Wd(x-x')\dx\dx'
        \bigg|
        \\
        &=
        \bigg|
        \int_{\Od}\int_{\Od}\int_0^1 [\nabla\phi(x+t(x'-x))-\nabla\phi(x)]\,\mathrm{d}t\cdot(x'-x)\qnl_{\delta}(x,x')\Wd(x-x')\dx\dx'
        \bigg|
        \\
        &\leq
        K_{2,n}^{-1} \|\qnl_\delta\|_{L^{1,2}(\Od\times\Od)}
        \sup_{|x-x'|<\delta} |\nabla\phi(x')-\nabla\phi(x)|,
      \end{aligned}
  \end{equation*}
  where we used Taylor's theorem, H{\"o}lder's inequality and the normalization of \(\Wd\) to arrive at the last term.
  To conclude the proof, it remains to note that \(\nabla\phi\) is a continuous function with compact support, and consequently is also uniformly continuous,
  and therefore
  \[
  \bigg|\int_{\O} \nabla\phi(x)\cdot \mathrm{d}\hat{q}(x)+\ell(\phi)\bigg|
  =\lim_{\delta\searrow 0}
  \bigg|\int_{\O} \nabla\phi(x)\cdot R_{\delta}\qnl_{\delta}(x)\dx+\ell(\phi)\bigg|
  = 0.
  \]
\end{proof}

We now note that Proposition~\ref{prop:claim1_new} is not quite sufficient to establish the feasibility of limit points of~\eqref{eq:prob_nl} in~\eqref{eq:limiting_local_dual_rig}, since we do not know whether it holds for all \(\phi \in C^1_0(\cl\O)\) and not just for all \(\phi \in C^1_c(\O)\).
If it were, then together with Proposition~\ref{prop:claim0} it would imply the inequality \begin{equation}\label{liminf0}p^*_{\text{loc}} \leq \liminf_{\delta\to 0} p^*_{\delta,a}.\end{equation}
Interestingly enough, this inequality is valid, as we shall see in the following section,
in which we consider the sequence of relaxations of~\eqref{eq:prob_nl}.
Ultimately, the relaxation property~\eqref{eq:dual_finite} together with Proposition~\ref{prop:1sidenew} (see also Theorem~\ref{thm:nonlocal_approximation})  imply~\eqref{liminf0}.

\section{The cost of flux anti-symmetry}\label{sec:gammaconv3}

Section~\ref{sec:analysis} uncovers several desirable properties of~\eqref{eq:prob_nl}.
Indeed, it appears naturally and rigorously as a vanishing material limit of nonlocal design problems, see Subsection~\ref{sec:gammaconv1},
and possesses solutions in Lebesgue spaces, see Subsection~\ref{sec:existence}.
We have also mentioned, and shall prove it in this section, that~\eqref{eq:prob_nl} even provide a one-sided estimate for the local problem~\eqref{eq:limiting_local_dual_rig} when the non-local interaction horizon goes to zero, see~\eqref{liminf0}.
Whether the corresponding limsup inequality holds for the problems~\eqref{eq:prob_nl} and~\eqref{eq:limiting_local_dual_rig} we do not know.
Nevertheless, in this section we present a very simple and surprising trade-off proposal. Indeed, if we are willing to sacrifice the requirement for the two-point fluxes to be anti-symmetric, and therefore also the results in Subsections~\ref{sec:existence} and~\ref{sec:gammaconv1}, we recover a two-sided estimate with respect to the diminishing non-local horizon.

We begin by stating the problem obtained from~\eqref{eq:prob_nl} by formally dropping the requirement of anti-symmetry of two-point fluxes:
\begin{equation}\label{eq:prob_nl_nosym}
    p^*_{\delta} = \inf_{\qnl \in \Qdx(f)} \|\qnl\|_{L^{1,2}(\Od\times\Od)} =
    \inf_{\qnl \in \Qdx} [\|\qnl\|_{L^{1,2}(\Od\times\Od)} + \ind_{\{f\}}(\diverx \qnl)].
\end{equation}
This problem is close to~\eqref{eq:limiting_local_dual}, in the sense that it is posed in an \(L^1\)-space, \(L^1(\Od;L^2(\Od))\) in the present case, and we cannot rely on the anti-symmetry of fluxes any longer, as we did in Subsection~\ref{sec:existence}.  Therefore, to assert existence of solutions to~\eqref{eq:prob_nl_nosym}, we would need to study its relaxation to the space of Radon measures, as was done for~\eqref{eq:limiting_local_dual} to arrive at~\eqref{eq:limiting_local_dual_rig}. We do not explore this route any further, and instead note that similarly to~\eqref{eq:prob_nl}, \eqref{eq:prob_nl_nosym} fulfills the conditions for the strong Fenchel duality,
see for example~\cite[Theorem~4.4.3]{borwein2004techniques}: indeed \(\|\cdot\|_{L^{1,2}(\Od\times\Od)}\)
is convex and continuous on the Banach space \(\Qdx\), \(\ind_{\{f\}}\) is proper, convex and l.s.c.\ on
the Hilbert space \(L^2(\O)\), and \(\diverx: \Qdx \to L^2(\O)\) is linear, bounded,
and surjective (Propositions~\ref{prop:diver0} and~\ref{prop:closable}).
Recalling that \(\ind^*_{\{f\}}(u) = -\ell(u)\), we proceed to explicitly computing the conjugate of the remaining part of the objective  of~\eqref{eq:prob_nl_nosym} as follows. For each \(u\in L^2(\O)\) we can write
\begin{equation}\label{eq:conj1}\begin{aligned}
\|\diverx^*u\|_{L^{1,2}(\Od\times\Od)}^*
&= \sup_{\qnl \in \Qdx}[ -(u, \diverx\qnl)_{L^2(\O)} - \|\qnl\|_{L^{1,2}(\Od\times\Od)}]
\\&= \ind_{B_{L^{\infty,2}(\Od\times\Od)}}(\grad u),
\end{aligned}\end{equation}
where Proposition~\ref{prop:adj}, point 2, is utilized to arrive at the final expression.
With this in mind we have the following dual of~\eqref{eq:prob_nl_nosym}, see~\cite[Theorem~4.4.3]{borwein2004techniques}:
\begin{equation}\label{eq:limiting_nl_dual_nosym}
  -p^*_{\delta}=-d^*_{\delta}\overset{\cdot}{=}\inf_{u\in L^2(\O)}[ \ind_{B_{L^{\infty,2}(\Od\times\Od)}}(\grad u) - \ell(u)],
\end{equation}
Furthermore, the infimum in~\eqref{eq:limiting_nl_dual_nosym} is attained since
\(d^*_{\delta}\) is finite:
\begin{equation}\label{eq:dual_finite}
   0\leq d^*_{\delta} = p^*_{\delta} \leq p^*_{\delta,a}
   =d^*_{\delta,a} \leq d^*_{\delta,2} < \infty,
\end{equation}
where the first inequality is owing to the inclusion \(0\in L^2(\O)\),
and the rest is owing to the obvious inclusion \(\Qdx^a\subset\Qdx\) and the previously established results; see Proposition~\ref{prop:infima_ineq}.

We will now establish two bounds on the optimal value for the local problem~\eqref{eq:limiting_local_predual_relax} in terms of those for the nonlocal problem~\eqref{eq:limiting_nl_dual_nosym}.
We begin with the following estimate, which provides us with the ``recovery sequence'' for establishing \(\Gamma\)-convergence.
\begin{proposition}\label{prop:1sidenew}
For each \(u\in W^{1,\infty}_0(\O)\), \(\|\nabla u\|_{L^{\infty}(\O)}\leq 1\),
there is a sequence  \(v_k \in C^\infty_c(\O)\) and a sequence \(\delta_k \searrow 0\), such that
\(\|\ddot{\mathcal{G}}_{\delta_k} v_k\|_{L^{\infty,2}(\Od\times\Od)}\leq 1\) and \(\|u-v_k\|_{L^\infty(\O)}\to 0\).
In particular, \begin{equation}\label{eq:recoverydual}\liminf_{\delta\searrow 0} d^*_\delta \geq \overline{pd}^*_{\text{loc}}.\end{equation}
\end{proposition}
\begin{proof}
Let us fix a Lipschitz representative of \(u\in W^{1,\infty}_0(\O)\) as described in the proposition and a positive sequence \(\epsilon_k\searrow 0\).
As in Proposition~\ref{prop:pdeq}, we use \cite[Remark~4.1]{deville2019approximation} to select
  \(v_k \in C^\infty_c(\O)\) such that the Lipschitz constant for \(v_k\) on \(\cl\O\) is no greater
  than the Lipschitz constant for \(u\), which equals to 1, and \(\|u - v_{\epsilon}\|_{C^0(\cl\O)} \leq \epsilon_k\).
  Without any loss of generality we can assume that the Lipschitz constant for \(v_{k}\) is \emph{strictly} smaller than one; it is sufficient to multiply \(v_{k}\) by \(1-\epsilon_k\).
  We now use Taylor's theorem to estimate
  \[
  \begin{aligned}
  \grad v_{k}(x,x')
  &= [v_{k}(x)-v_{k}(x')]\Wd(x-x')
  \\&= \bigg[\nabla v_{k}(x) + \frac{1}{2}\nabla^2 v_{k}(\xi)(x-x')\bigg]\cdot (x-x')\Wd(x-x'),
  \end{aligned}
  \]
  where \(\xi\in\Rn\) is a point between \(x\) and \(x'\).
  We can now proceed as in~\eqref{eq:psi} to estimating
  \[
  \begin{aligned}
    \|  \grad v_{k} \|_{L^{\infty,2}(\Od\times\Od)}
    \leq \|\nabla v_{k}\|_{L^\infty(\O;\Rn)}
    + O(\delta),
  \end{aligned}
  \]
  where the constant in the last term depends on \(\|\nabla^2 v_{k}\|_{L^\infty(\O;\R^{n\times n})}\).
  By our construction \(\|\nabla v_{k}\|_{L^\infty(\O;\Rn)} < 1\), and consequently for each \(\epsilon_k >0\) we can choose \(\delta_k>0\),
  where without loss of generality we can assume that \(\delta_k\) is monotonically decreasing,
  such that \(\|  \ddot{\mathcal{G}}_{\delta_k} v_{k} \|_{L^{\infty,2}(\O_{\delta_k}\times\O_{\delta_k})} \leq 1\).

  The last claim follows from observing the fact that each \(v_k\) is feasible for~\eqref{eq:limiting_nl_dual_nosym} with \(\delta=\delta_k\),
  and that \(\lim_{k\to\infty} \ell(v_k) = \ell(u)\).
\end{proof}

Note that Proposition~\ref{prop:1sidenew} additionally provides us with the following estimate for problems~\eqref{eq:prob_nl} discussed in the previous section.
Indeed, we have the inequalities
\begin{equation}
\liminf_{\delta\searrow 0} p^*_{\delta,a}
\geq
\liminf_{\delta\searrow 0} p^*_{\delta}
=
\liminf_{\delta\searrow 0} d^*_\delta \geq \overline{pd}^*_{\text{loc}}
= p^*_{\text{loc}},
\end{equation}
where equalities are owing to the strong duality, and the first inequality is owing to the fact that~\eqref{eq:prob_nl_nosym} is a relaxation of~\eqref{eq:prob_nl}.

We now estimate the optimal value of~\eqref{eq:limiting_local_predual_relax} from the other side.  Together with~\eqref{eq:recoverydual}, these will be sufficient for establishing \(\Gamma\)-convergence.
This inequality hinges on the generalization of the well established estimates found in~\cite{bourgain2001another,ponce2004estimate,ponce2004new} to the present situation of Lebesgue norms with mixed exponents.

\begin{proposition}\label{prop:compactness}
For each \(0<\delta<\bar{\delta}\), let \(u_\delta\in L^2(\O)\) be such that \(\|\grad u_\delta\|_{L^{\infty,2}(\Od\times\Od)}\le 1\). Then
the family \(\{u_\delta\}_{0<\delta<\bar{\delta}}\) is relatively compact in \(L^2(\Omega)\). Furthermore, if \(u_{\delta_j}\to u\) in \(L^2(\O)\), with \(\delta_j\to 0\) as \(j\to \infty\), then \(u\in W^{1,\infty}_0(\O)\) and \(\|\nabla u\|_{L^\infty(\O)}\le 1\).
\end{proposition}
\begin{proof}
First of all, we consider functions \(u_\delta\) extended by zero outside \(\O\).
Since \(\Wd\) vanishes on \(\Rn\setminus \delta B_{\Rn}\), then
  \[\| \ddot{\mathcal{G}}_\delta u_\delta\|_{L^{\infty,2}(\O_{\bar\delta}\times \O_{\bar\delta})}  =\|\ddot{\mathcal{G}}_\delta u_\delta\|_{L^{\infty,2}(\Od\times\Od)}\le 1,\]
  for any \(\delta\in (0,\bar{\delta})\). Moreover, owing to the estimate
  \begin{equation}\label{l2bound}\begin{aligned}\int_{\O_{\bar\delta}}\int_{\O_{\bar\delta}}\frac{|u(x)-u(x')|^2}{|x-x'|^2} |x-x'|^2 \omega_\delta^2(x,x')\,dx'\,dx&\le |\Omega_{\bar\delta}|\|\ddot{\mathcal{G}}_\delta u_\delta\|^2_{L^{\infty,2}(\Od\times\Od)}\\&\le  |\Omega_{\bar\delta}|,\end{aligned}\end{equation}
      we can utilize~\cite[Theorem 1.2]{ponce2004estimate} to conclude that the family \(\{u_\delta\}_{0<\delta<\bar{\delta}}\) is relatively compact in \(L^2(\Omega_{\bar\delta})\)
      with limit points in \(W^{1,2}(\Omega_{\bar\delta})\).

Let us fix an arbitrary \(p\in (2,+\infty)\).
For some \(\phi\in C_0^{\infty}(B_{\Rn})\), \(\phi\ge 0\), \(\int_{\Rn} \phi(x)\,dx=1\), let us constract a family of mollifiers \(\phi_\epsilon(x)=\epsilon^{-n} \phi(\epsilon^{-1}x)\). Then, for each \(\delta\in (0,\bar{\delta})\) and \(\epsilon \in (0,\bar{\delta}-\delta)\),
we have the inequalities
\begin{align*}
&\int_{\O_{\bar\delta}}\left(\int_{\O_{\bar\delta}} {|\phi_\epsilon\ast u_\delta(x)-\phi_\epsilon\ast u_\delta(x')|^2}\Wd^2(x-x')\,dx'\right)^\frac{p}{2}\,dx \\
&\le \int_{\O_{\bar\delta}}\left(\int_{\O_{\bar\delta}} \int_{\epsilon B_{\Rn}}\phi_\epsilon(z){| u_\delta(x-z)- u_\delta(x'-z)|^2}\Wd^2(x-x')\,dz\,dx'\right)^\frac{p}{2}\,dx\\
&=\int_{\O_{\bar\delta}}\left(\int_{\epsilon B_{\Rn}}\phi_\epsilon(z)\int_{\O_{\bar\delta}} {| u_\delta(x-z)- u_\delta(x'-z)|^2}\Wd^2(x-x')\,dx'\,dz\right)^\frac{p}{2}\,dx\\
&\le\int_{\O_{\bar\delta}}\int_{\epsilon B_{\Rn}}\phi_\epsilon(z)\left(\int_{\O_{\bar\delta}} {| u_\delta(x-z)- u_\delta(x'-z)|^2}\Wd^2(x-x')\,dx'\right)^\frac{p}{2}\,dz\,dx\\
&=\int_{\epsilon B_{\Rn}}\phi_\epsilon(z)\int_{\O_{\bar\delta}}\left(\int_{\O_{\bar\delta}} {| u_\delta(x-z)- u_\delta(x'-z)|^2}\Wd^2(x-x')\,dx'\right)^\frac{p}{2}\,dx\,dz\\
&=\int_{\O_{\bar\delta}}\left(\int_{\O_{\bar\delta}} {| u_\delta(x)- u_\delta(x')|^2}\Wd^2(x-x')\,dx'\right)^\frac{p}{2}\,dx
\\&\le \|\ddot{\mathcal{G}}_\delta u_\delta\|_{L^{\infty,2}(\Od\times\Od)}^{p-2}
\|\ddot{\mathcal{G}}_\delta u_\delta\|^2_{L^{2}(\Od\times\Od)}
\le  |\Omega_{\bar\delta}|,
\end{align*}
where Jensen's inequality has been applied in the first two inequalities above,
Fubini's theorem in the first two equalities and a double translation of variables
has been performed for the last equality.
The last two inequalities follow from H{\"o}lder's inequality and~\eqref{l2bound}.

In the proof of~\cite[Theorem 2]{bourgain2001another} it is shown that
for each \(u\in C^2(\cl\O_{\bar\delta})\) and \(x\in \O_{\bar\delta}\)
the identity
\[\lim_{\delta\searrow 0}\int_{\O_{\bar\delta}} \frac{|u(x)- u(x')|^2}{|x-x'|^2}|x-x'|^2\Wd^2(x-x')\,dx'=|\nabla u(x)|^2\]
holds. Furthermore, owing to the uniform bound
\[\begin{aligned}&\left(\int_{\O_{\bar\delta}} {| u(x)-u(x')|^2}\Wd^2(x-x')\,dx'\right)^\frac{p}{2}\\&\le \|\nabla u\|^p_{C^0(\cl{\O_{\bar\delta}})}\left(\int_{\O_{\bar\delta}} |x-x'|^2 \Wd^2(x-x')\,dx'\right)^\frac{p}{2}\le K^{-\frac{p}{2}}_{2,n}\|\nabla u\|^p_{C^0(\cl{\O_{\bar\delta}})}, \end{aligned}\]
we can apply Lebesgue dominated convergence theorem to compute
\[\lim_{\delta\searrow 0}\int_{\O_{\bar\delta}}\left(\int_{\O_{\bar\delta}} \frac{| u(x)- u(x')|^2}{|x-x'|^2}|x-x'|^2\omega_\delta(x-x')\,dx'\right)^\frac{p}{2}\,dx= \|\nabla u\|_{L^p(\O_{\bar\delta};\Rn)}^p.\]
One should also note that this limit is uniform on bounded sets of \(C^2(\cl{\O_{\bar\delta}})\).
Therefore taking limit as \(\delta\searrow 0\) in the inequality
\[\int_{\O_{\bar\delta}}\left(\int_{\O_{\bar\delta}} \frac{|\phi_\epsilon\ast u_\delta(x)-\phi_\epsilon\ast u_\delta(x')|^2}{|x-x'|^2}|x-x'|^2\omega_\delta(x-x')\,dx'\right)^\frac{p}{2}\,dx\le |\O_{\bar{\delta}}|\]
we arrive at
\[\|\nabla (\phi_\epsilon\ast u)\|_{L^p(\O_{\bar\delta};\Rn)}^p\le |\O_{\bar\delta}|.\]
Taking into account that \(u\in W^{1,2}(\O_{\bar\delta})\), it holds that
\(\nabla (\phi_\epsilon\ast u)(x)\) converges to \(\nabla u(x)\) for almost all
\(x\in \O_{\bar{\delta}}\)  as \(\epsilon\searrow 0\). Consequently owing to Fatou's lemma
we get the estimate
\[ \| \nabla u\|_{L^p(\O_{\bar{\delta}};\Rn)}^p\le \limsup_{\epsilon\searrow 0}\|\nabla (\phi_\epsilon\ast u)\|_{L^p(\O_{\bar\delta};\Rn)}^p\le |\O_{\bar\delta}|.\]
To complete the proof it remains to compute the limit
\[\|\nabla u\|_{L^\infty(\O_{\bar{\delta}};\Rn)} = \lim_{p\to\infty}
\| \nabla u\|_{L^p(\O_{\bar{\delta}};\Rn)} \leq \lim_{p\to\infty}|\O_{\bar\delta}|^{1/p} = 1.\]
We then  use the fact that \(u\equiv 0\) on \(\O_{\bar{\delta}}\setminus \O\) and argue as in Proposition~\ref{prop:tikh_exist_loc}, point 5,
to conclude that \(u \in W^{1,\infty}_0(\O)\).
\end{proof}

\begin{theorem}\label{thm:nonlocal_approximation}
  Problem  \eqref{eq:limiting_local_predual_relax} is a \(\Gamma\)-limit of problems \eqref{eq:limiting_nl_dual_nosym},
  as \(\delta\searrow 0\).
\end{theorem}
\begin{proof}
  Proposition~\ref{prop:1sidenew} provides us with the \(\limsup\)-inequality for \(\Gamma\)-convergence and a suitable recovery sequence.
  Proposition~\ref{prop:compactness} asserts that sequences of feasible points for~\eqref{eq:limiting_nl_dual_nosym} are relatively compact in \(L^2(\O_{\bar{\delta}})\) with limit points, which are feasible in~\eqref{eq:limiting_local_predual_relax}.
  Note that the objective functions in all these problems, when finite, reduce to \(\ell\), which is continuous with respect to convergence in \(L^2(\O_{\bar{\delta}})\).  This is sufficient to conclude, that \(\liminf\)-inequality for \(\Gamma\)-convergence holds.
  It remains to appeal to the equivalent definition of \(\Gamma\)-convergence, see for example~\cite[Theorem~2.1]{braides2006handbook}.
\end{proof}

We can see that the difference between the results in Section~\ref{sec:analysis} and the present one boils down to assuming, or not assuming the anti-symmetry of fluxes, or equivalently the existence of the gap between \(d_{\delta,a}^*\) and \(d_{\delta}^*\).
We would like to conclude this work by discussing the issue of anti-symmetry with the help of an example, which shows that the values of the functions \(\|\cdot\|_{L^{\infty,2}(\O\times\O)}\) and \(\|\cdot\|^*_{L^{1,2}_a(\O\times\O)}\) participating in the definition of the problems~\eqref{eq:limiting_nl_dual_nosym} and~\eqref{eq:limiting_nl_dual_asym}, may differ significantly even on anti-symmetric functions.
It is therefore possible that the optimal values \(d_{\delta,a}^*\) and \(d_{\delta}^*\) differ for each \(\delta>0\), but we do not know what happens in the limit \(\delta\searrow 0\).
\begin{example}
In this example we consider \(\O=(0,1)\subset \R^1\), but the same conclusions may be reached in higher-dimensions as well. (In fact, we could simply assume an ``extruded'' situation, where functions are constant along other coordinate directions.)
Let \(\pnl \in L^{\infty,2}_a(\O\times\O)\) be given by \(\pnl(x,x') = -\sin(\pi x) + \sin(\pi x')\).
A direct computation shows that \(\|\pnl\|_{L^{\infty,2}(\O\times\O)} = 2^{-1/2}\).
We will now approximately compute the two dual norms
\begin{align}
\label{eq:dual_norm_nosym}
\|\pnl\|_{L^{\infty,2}(\O\times\O)} &= \sup_{\qnl\in B_{L^{1,2}(\O\times\O)}} \langle\pnl,\qnl\rangle, \quad\text{and}\\
\|\pnl\|^*_{L^{1,2}_a(\O\times\O)} &= \sup_{\qnl\in B_{L^{1,2}_a(\O\times\O)}} \langle\pnl,\qnl\rangle,
\label{eq:dual_norm_asym}
\end{align}
which determine the feasible sets in the problems~\eqref{eq:limiting_nl_dual_nosym} and~\eqref{eq:limiting_nl_dual_asym}.
We subdivide of \(\O\times\O\) into \(N\times N\) equal square cells, and let the suprema in the  expressions above be computed over
piecewise-constant splines respecting the subdivision.
In this way, the computation transforms naturally into solving two linear programs with linear and second-order conic constraints, which can be solved by most modern interior point solvers; we use SDPT3~\cite{sdpt3} interfaced through CVX~\cite{cvx}.
The results are shown in Figure~\ref{fig1}, where
the gap between the two suprema is quite apparent.
The optimal solutions to the discretized versions of~\eqref{eq:limiting_nl_dual_nosym} and~\eqref{eq:limiting_nl_dual_asym} also look very distinct.
Without assuming anti-symmetry the supremum is not attained in \(L^{1,2}(\O\times\O)\), and the discrete solutions lose anti-symmetry and become increasingly concentrated along two boundaries of \(\O\times\O\), consistent with the fact that these extremal problems in \(L^1(\O;L^2(\O))\) should be further relaxed into the space of Radon measures.
At the same time, when the anti-symmetry is assumed, the corresponding optimization problem seems to lose its sparcity-encouraging property.
\begin{figure}[htb]
    \centering
    \begin{tabular}{cc}
    \includegraphics[width=0.45\textwidth]{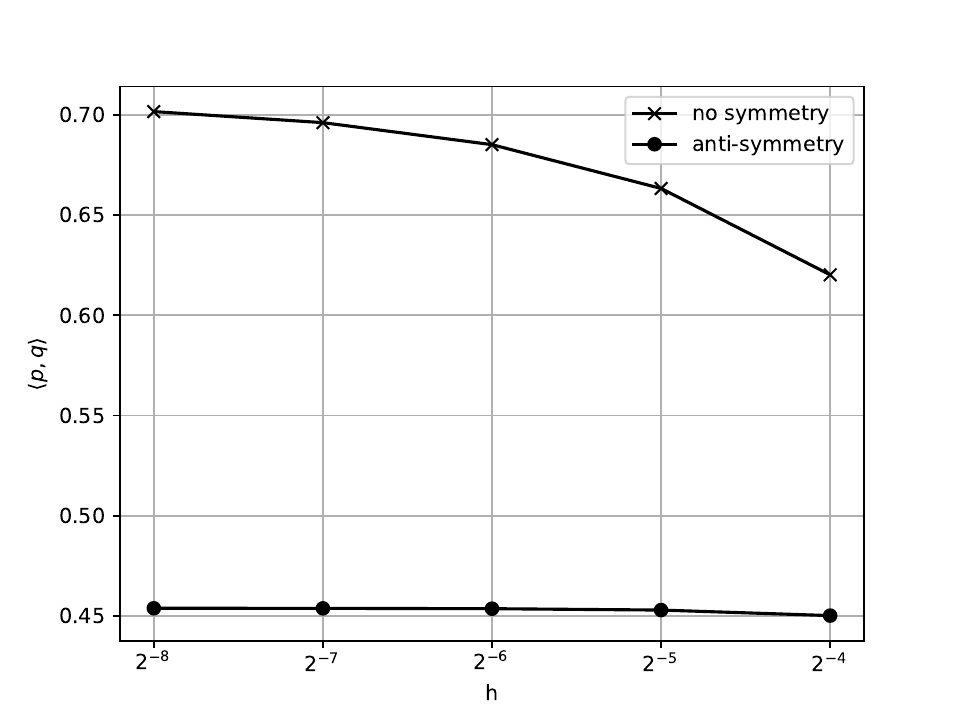}&
    \includegraphics[width=0.45\textwidth]{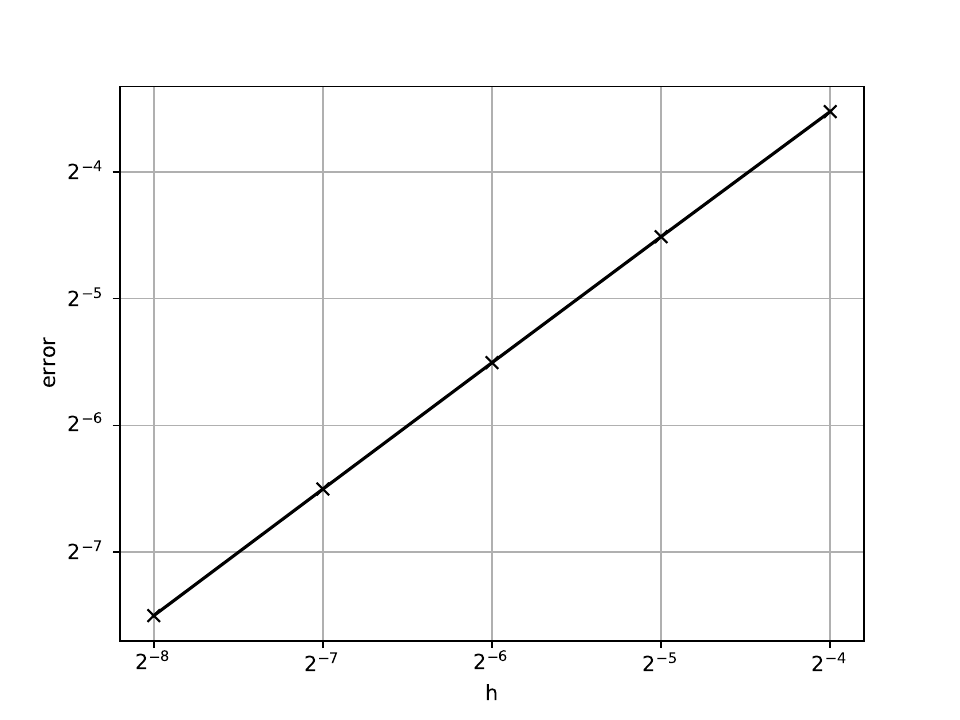}\\
    (a)&(b)\\
    \includegraphics[width=0.45\textwidth]{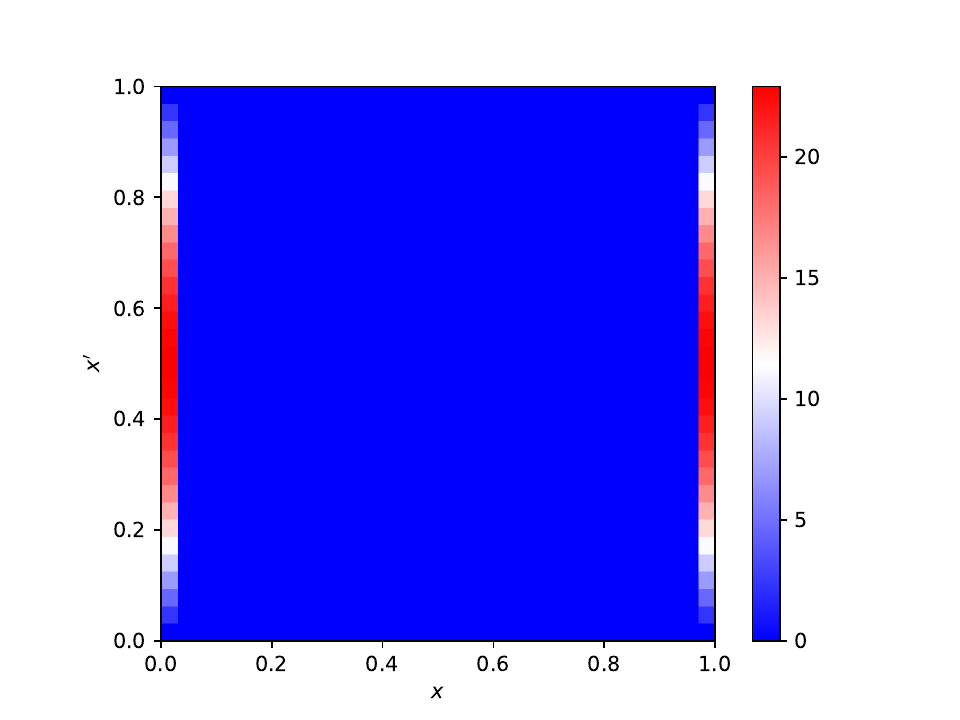}&
    \includegraphics[width=0.45\textwidth]{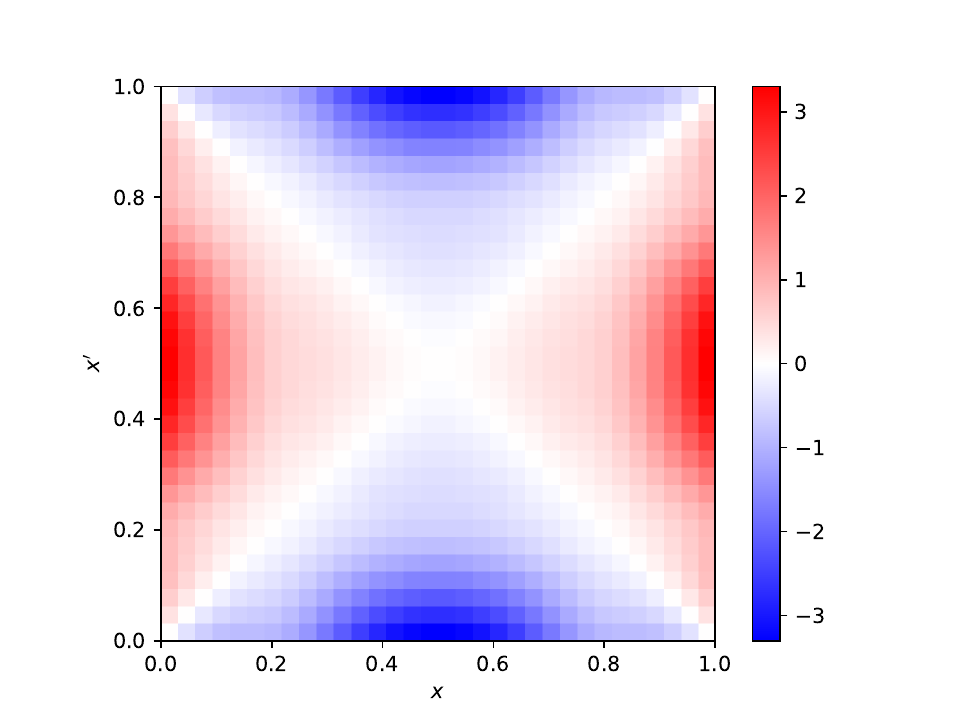}\\
    (c)&(d)
    \end{tabular}
    \caption{Approximate computation of the dual norm, with or without assuming anti-symmetry.
    (a): Approximate value of suprema as a function of the cell size.
    The gap between the two suprema is apparent.
    (b): Discrepancy between the known analytical value of the norm and its numerical approximation, as a function of the cell size.
    (c): Optimal solution \(\qnl\in B_{L^{1,2}(\O\times\O)}\) to the discretization of~\eqref{eq:dual_norm_nosym} on a \(32\times 32\) grid.
    (d): Optimal solution \(\qnl\in B_{L^{1,2}_a(\O\times\O)}\) to the discretization of~\eqref{eq:dual_norm_asym} on a \(32\times 32\) grid.}
    \label{fig1}
\end{figure}
\end{example}

\section*{Acknowledgement}
AE is grateful to Jens Gravesen and Christian Henriksen for a discussion of Lebesgue spaces with mixed
exponents, and to Jakob Lemvig for pointing out the reference~\cite{benedek1961space}.
JCB acknowledges financial support from Spanish Agencia Estatal de Investigaci\'on through project PID2020-116207GB-I00 and Junta de Comunidades de Castilla-La Mancha through project  SBPLY/19/180501/000110.

\bibliographystyle{plain}
\bibliography{nloc_sparsebasis}

\end{document}